\documentclass[11pt]{article}

\usepackage{paper,macros,graphicx,epstopdf}

\def\d{\delta} 
\def\ds{\displaystyle} 
\def\e{{\epsilon}}

\renewcommand{\r}{{\bf R}} 

\title{Stochastic Optimization Using a Trust-Region Method and Random Models}\author{
R. Chen\thanks{ {\tt Ruobing.Chen@us.bosch.com}, Bosch Research. The work of this author was
 partially supported by NSF Grant CCF-1320137 and 
 AFOSR Grant FA9550-11-1-0239.}
\and
M. Menickelly \thanks{{\tt mjm412@lehigh.edu}, Department of Industrial and Systems Engineering, Lehigh University,
Harold S. Mohler Laboratory, 200 West Packer Avenue, Bethlehem, PA 18015-1582, USA. The work of this author is
partially supported by NSF Grants DMS 13-19356 and CCF-1320137}
\and
K. Scheinberg\thanks{{\tt katyas@lehigh.edu}
Department of Industrial and Systems Engineering, Lehigh University,
Harold S. Mohler Laboratory, 200 West Packer Avenue, Bethlehem, PA 18015-1582, USA
({\tt katyas@lehigh.edu}). The work of this author is partially supported by NSF Grants DMS 10-16571, DMS 13-19356, CCF-1320137,
 AFOSR Grant FA9550-11-1-0239, and  DARPA grant FA 9550-12-1-0406 negotiated by AFOSR.}
}

\begin{document}

\maketitle


\begin{abstract}
In this paper, we propose and analyze a trust-region model-based algorithm for solving unconstrained stochastic optimization problems. 
Our framework utilizes random models of an objective function $f(x)$, obtained from stochastic observations of the function or its gradient. Our method also utilizes estimates of function values to gauge progress that is being made. The convergence analysis relies on requirements that these models and these estimates are sufficiently accurate with high enough, but fixed, probability. Beyond these conditions, no assumptions are made on how these models and estimates are generated. Under these general conditions we show an almost sure global convergence of the method to a first order stationary point. In the second part of the paper, we present examples of generating sufficiently accurate random models under biased or unbiased noise assumptions. Lastly, we present some computational results showing the benefits of the proposed method compared to existing approaches that are based on sample averaging or stochastic gradients. 
\end{abstract}

\section{Introduction}\label{sec.introduction}
Derivative free optimization (DFO) \cite{DFObook} has recently grown as a field of nonlinear optimization which addressed optimization of black-box functions, that is functions whose value can be (approximately) computed by some numerical procedure or an experiment, while their closed-form expressions and/or derivatives are not available and cannot be approximated accurately or efficiently. Although the role of derivative-free optimization is particularly important when objective functions are noisy, traditional DFO  methods have been developed primarily for deterministic functions. The fields of stochastic optimization \cite{RobbinsMonro,RuszShapiroBook} and stochastic approximation \cite{SpallBook} on the other hand 
 focus on optimizing functions that are stochastic in nature. Much of the focus of these methods depend on the availability and use of stochastic derivatives, however, some work has addressed stochastic black box functions, typically by some sort of a finite differencing scheme \cite{KieferWolfowitz}. 
 
 In this paper, using methods developed for DFO, we aim to solve
 \begin{eqnarray}\label{eq:object}
\ds\min\limits_{x\in \r^n} \,  f(x)
\end{eqnarray}
where  $f(x)$ is a function which is  assumed to be smooth and bounded from below, and whose value can only be computed with some noise. 
Let $\tilde{f}$ be the noisy computable version of $f$, which takes the form
$$\tilde{f}(x) = f(x,\varepsilon)$$
where the noise $\varepsilon$ is a random variable.
 
 In recent years, some DFO methods have been extended to and analyzed for stochastic functions \cite{DengFerris}. Additionally, stochastic approximation methodologies
 started to incorporate techniques from the DFO literature \cite{ChangHongWan}. 
 The analysis in all that work assumes some particular structure of the noise, 
 including the assumption that the noisy function values give an unbiased estimator of the true function value.

 There are two main classes of methods in this setting of stochastic optimization: stochastic gradient (SG) methods (such as the well known Robbins-Monro method) and sample averaging (SA) methods. The former (SG) methods roughly work as follows: they obtain a realization of an unbiased estimator of the gradient at each iteration and take a step in the direction of the negative gradient. The step sizes progressively diminish and the iterates are averaged to form a sequence that converges to a solution. These methods typically have very inexpensive iterations, but exhibit slow convergence and are strongly dependent on the choice of algorithmic parameters, particularly the sequence of step sizes. Many variants exist that average the gradient information from past iterations and are able to accept sufficiently small, but nondecreasing step sizes \cite{PolyakJuditsky, BachMoulines}. However, the convergence remains slow, and parameter tuning remains necessary in most cases.  
 Moreover, the majority of these methods have been developed exclusively for convex functions and may not converge in non convex settings.
An exception is the randomized stochastic (accelerated) gradient (RS(A)G) method presented in \cite{ghadimilanconvex}, \cite{ghadimilannonconvex}. This method employs random stopping criteria to provide theoretical first-order convergence even in nonconvex cases, but does not appear to be very practical.

 The second class of methods, (SA), is based on sample averaging of the function and gradient estimators,which is applied to reduce the variance of the noise. 
 These methods repeatedly
 sample the function value at a set of points in hopes to ensure sufficient accuracy of the function and gradient estimates. For a thorough introduction and references therein, see \cite{pasupathyinformstutorial}.
 The optimization method and sampling process are usually tightly connected in these approaches, hence, again, algorithmic parameters need to be specially chosen and tuned. 
 These methods tend to be more robust with respect to parameters and enjoy faster convergence at a cost of more expensive iterations. 
 However, none of these methods are applicable in the case of  biased noise and they suffer significantly in the presence of outliers.

  The goal of this paper
 is to show that a {\em standard efficient unconstrained optimization method}, such as a trust region method, can be applied, with very small modifications, to 
 stochastic nonlinear (not necessarily convex) functions and can be guaranteed to converge to  first order stationary points
 as long as certain conditions are satisfied. 
We present a general framework, where we do not specify any particular sampling technique. 
 The framework is based on the trust region DFO framework \cite{DFObook}, and its extension to probabilistic models \cite{Bandeiraetal2013}. 
In terms of this framework and the certain conditions that must be satisfied, 
 we essentially assume that 
 \begin{itemize}
 \item
  the local models of the objective function constructed on each iteration satisfy some first order accuracy requirement with sufficiently high probability, 
 \item and  that function estimates at
 the current iterate and at a potential next iterate are sufficiently accurate with sufficiently high probability. 
 \end{itemize}
 The main novelty of this work is the analysis of the framework and the resulting weaker, more general, conditions compared to prior work. 
 In particular,
 \begin{itemize}
 \item we do not assume  that the probabilities of obtaining sufficiently accurate models and estimates are increasing 
 (they simply need to be above a certain constant) and
 \item  we do not assume any distribution of the random models and estimates.
  In other words, if a model or estimate is inaccurate, it can be arbitrarily inaccurate, i.e. the noise in the function values can have nonconstant bias. 
  \end{itemize}
  It is also important to note that while our framework and model requirements are 
  borrowed from prior work on DFO, this framework  applies to derivative-based optimization as well. Later in the paper we will discuss different settings which will 
 fit into the proposed framework.

 This paper consists of two main parts. In the first part we  propose and analyze a trust region framework, which utilizes random models of $f(x)$ at each iteration to compute the next potential iterate.  It also relies on (random, noisy) estimates of the function values at the current iterate and the potential iterate to gauge the progress that is being made. The convergence analysis then relies on requirements that these models and these estimates are sufficiently accurate with sufficiently high probability. 
 Beyond these conditions, no assumptions are made about how these models and estimates are generated. The resulting method is a stochastic process that is analyzed with the help of martingale theory.  The method is shown to converge to first order stationary points with probability one. 
 
  In the second part of the paper we consider various scenarios under different assumptions on the   noise inducing component $\varepsilon$ and 
  discuss how  sufficiently accurate random models 
  can be generated. In particular, we show that in the case of unbiased noise, that is when 
 $ \mathbb{E}[{f}(x,\varepsilon)] = f(x)$ and  ${\rm Var}[f(x,\varepsilon)]\leq\sigma^2< \infty $ for all $x$,  sample averaging techniques give us 
 sufficiently accurate models. Although we will prove convergence under the mentioned framework that essentially says we have the ability to compute both sufficiently accurate models and estimates with constant, separate probabilities, it is not necessarily easy to estimate what these probabilities ought to be for a given problem. While we provide some guidance on the selection of sampling rates in an unbiased noise setting in Section 5, our numerical experiments show that the bounds on probabilities suggested by our theory to be necessary for almost sure convergence are far from tight. 

 We also discuss the case where $ \mathbb{E}[{f}(x,\varepsilon)] \neq f(x)$, and where the noise bias may depend on $x$ or 
 on the method of computation of the function values. 
     One simple setting, which is illustrative, is as follows. 
 Suppose we have an objective function, which is computed by a numerical process, whose accuracy can be controlled 
 (such as tightening a convergence criterion within this numerical process). Suppose now that this numerical process involves some random component (such as sampling from some large data and/or utilizing a randomized algorithm). It may be known that with sufficiently high probability this numerical process produces a sufficiently accurate 
 function value - however, with some small (but nonzero) probability the numerical process may fail and hence no reasonable value is guaranteed. 
 Moreover, such failures may become more likely as more accurate computations are required (for instance because an upper bound on the total number of iterations is reached inside the numerical process). Hence the probability of failure may depend on the current iterate and state of the algorithm. Here we simply assume that such failures do not occur with probability higher than some constant (which will be specified in our analysis), conditioned on the past. However, we do not assume anything about the inaccurate function values.  As we will demonstrate  later in this paper, in this setting, $ \mathbb{E}[{f}(x,\varepsilon)] \neq f(x)$. 

   \subsection{Comparison with related work} 
 There is a very large volume of work on SA and SG, most of which is quite different from our proposed analysis and method. 
 However, we will mention a few works here that are most closely related to this paper and highlight the differences. The three  methods existing in the literature we will compare with are by
 Deng and Ferris \cite{DengFerris}, SPSA (simultaneous perturbations stochastic approximation) \cite{SPSA,2SPSA}, and SCSR (sampling controlled stochastic recursion) \cite{Pasupathyetal}. These three settings and methods are most closely related to our work
 because they all rely on using models of the objective function that can both incorporate second-order information and whose accuracy with respect to a ``true" model
 can be dynamically adjusted. In particular,  Deng and Ferris  apply the trust-region model-based derivative free optimization method UOBYQA \cite{PowellUOBYQA} in a setting of sample path optimization \cite{Robinson}. In \cite{SPSA} and    \cite{2SPSA}, the author applies an approximate gradient descent and Newton method, respectively,
 with  gradient and Hessian estimates computed from specially designed finite differencing techniques, with decaying finite differencing parameter. In \cite{Pasupathyetal} a very general scheme is presented, 
 where various deterministic optimization algorithms are generalized as stochastic counterparts, with the stochastic component arising from the stochasticity of the 
 models and the resulting step of the optimization algorithm. We now compare some key components of these three methods with those of our framework, which we hereforth refer to as STORM (STochastic  Optimization with Random Models).

 {\bf Deng and Ferris:}  The assumptions of the sample path setting are roughly as follows: on each iteration $k$, 
 given a collection of points $X^k=\{ x_1^k, \ldots, x_p^k\}$ 
 one can compute noisy function values $f(x_1^k, \varepsilon ^k),\ldots,  f(x_p^k, \varepsilon ^k) $.
 The noisy function values are assumed to be realizations of an unbiased estimator of true values $f(x_1^k),\ldots,  f(x_p^k)$.
 Then, using multiple, say $N_k$, realizations of $\varepsilon^k$, average function values $f^{N_k}(x_1^k),\ldots,  f^{N_k}(x_p^k) $ can be computed. 
 A quadratic model $m^{N_k}_k(x)$ is then fit into these function values, and so a sequence of models $\{m^{N_k}_k(x)\}$ is created using a nondecreasing sequence of sampling rates $\{N_k\}$. 
 The assumption on this sequence of  models  is that each of them satisfies a sufficient decrease condition (with respect to the true model of the true function $f$)
  with  probability $1-\alpha_k$, such that $\sum_{k=1}^\infty \alpha_k< \infty$. 
  The trust region maintenance follows the usual scheme like that in UOBYQA, hence the steps taken by the algorithm can be increased or decreased depending on the observed improvement of the function estimates. 
  
  {\bf SPSA:} The first order version of this method assumes that $f(x,\varepsilon)$ is an unbiased estimate of $f(x)$, and the second order version, 2SPSA, 
  assumes that an unbiased estimate of $\nabla f(x)$, $g(x,\varepsilon)\in \r^n$, can be computed.  Gradient (in the first order case) and Hessian (in the second order case) estimates are constructed using an interesting randomized finite differencing scheme. The finite difference step is assumed to be decaying to zero.
 An approximate steepest descent direction or approximate Newton direction are then constructed and a step of length $t_k$ is taken along this direction. The sequence $\{t_k\}$ is assumed to be decaying in the usual Robbins-Monro way, that is $t_k\to 0$, $\sum_k t_k =\infty$. 
  Hence, while no increase in accuracy of the models is assumed (they only need to be accurate in expectation), the step size parameter and the finite differencing parameter need to be tuned. Decaying step sizes often lead to slow convergence, as has been observed often in stochastic optimization literature. 
  
  {\bf SCSR:} This is a very general scheme which can include multiple optimization methods and sampling rates. The key ingredients of this scheme are a deterministic optimization method, and a stochastic variant that approximates it. The stochastic step (recursion) is assumed to be a sufficiently accurate approximation of the deterministic step  with  increasing probability (the probabilities of failure for each iteration are summable). This assumption is stronger than the one in this paper. In addition, another key assumption made for SCSR is that the iterates produced by the base deterministic algorithm converge to the unique optimal minimizer $x^*$.
  Not only we do not assume here that the minimizer/stationary point is unique, but we also do not assume a priori that the iterates form a convergent sequence, since they may not do so in a non convex setting, while every iterate subsequence converges to a stationary point. 
  
  {\bf STORM:}  Like the Deng and Ferris method, we utilize a trust-region, model-based framework, where the size of the trust region can be increased or decreased according to empirically observed function decrease and the size of the observed approximate gradients. The desired accuracy of the models is tied only to the trust region radius
  in our case, while for Deng and Ferris, it is tied to both the radius and the size of true model gradients (the second condition is harder to ensure). In either method, this desired accuracy is assumed to hold with some probability - in STORM, this probability remains constant throughout the progress of the algorithm, while for Deng and Ferris it has to  converge to $1$ sufficiently rapidly. 

There are three major advantages to our results. First of all, in the case of unbiased noise, the sampling rate is directly connected to the desired accuracy 
  of the estimates  and the probability with which this accuracy is achieved. Hence, for the STORM method, the sampling rate may increase or decrease according to the trust region radius, eventually increasing only when necessary, i.e. when the noise become dominating.  For all the other methods listed here, the sampling rate is assumed to increase monotonically. Secondly, in the case of biased noise, we can still prove convergence  of our method, as long as the desired accuracy is achieved with a fixed probability. In other words, we
  allow for the noise to be arbitrarily large with a small, but fixed probability, on each iteration. This allows us to consider new models of noise which cannot be handled by any of the other  methods discussed here. 
 In addition, STORM incorporates first and second order models without changing the algorithm - the step size parameter (i.e., the trust region radius) and other parameters of the method are chosen almost identically to the standard practices of the trust region methods, which have proved to be very effective in practice for unconstrained nonlinear optimization. In Section \ref{sec:computations} we show that the STORM method is very effective in different noise settings and is very robust with respect to sampling strategies.

Finally,  we want to point to an unpublished work  \cite{Larson2016}, which proposes  a very similar method to the one in this paper. Both methods were developed based on the trust region DFO method with random models for deterministic functions analyzed in \cite{Bandeiraetal2013} and extended to the stochastic setting. Some of the assumptions in this paper were inspired by an early version of \cite{Larson2016}. However, the assumptions and the analysis in \cite{Larson2016} are quite different from the ones in this paper. In particular, they rely on the assumption
that $f(x,\varepsilon)$ is an unbiased estimate of $f(x)$, hence their analysis does not extend to the biased case. Also they assume that the  probability of having an accurate model
at the $k$-th iteration is at least $1-\alpha_k$, such that $\alpha_k\to 0$, while for our method this probability can remain bounded away from zero. Similarly, they assume that the probability  of 
having accurate function estimates at the $k$-th iteration also converges to $1$ sufficiently rapidly, while in our case it is again constant. Their analysis, as a result, is very different from ours, and does not generalize to various stochastic settings (they only focus on the derivative free setting with additive noise).  The advantage of their method is that they do not need to put a restriction on acceptable step sizes, when the norm  of the gradient of the model is small. We, on the other hand, impose such restriction in our method and use it in the proof of our main result. However, as we discuss later in the paper
 this restriction can be relaxed at the cost of more complex algorithm and analysis. In practice, we do not implement this restriction, hence our basic  implementation is virtually identical to that in  \cite{Larson2016} except that we implement a variety of model building strategies, while only one such strategy 
 (regression models based on randomly rotated orthogonal samples sets) is implemented in   \cite{Larson2016}. Thus we do not directly compare the empirical performance of our method with the method in   \cite{Larson2016} since we view them as more or less the same method. 
 

  We conclude this section by introducing some frequently used notations and their meanings. The rest of the paper is organized as follows. In Section \ref{sec:TR-method}
  we introduce the trust region framework, followed by Section \ref{sec:models-estimates}, where we discuss the requirements on our random models and function estimates. The main convergence results are presented in Section \ref{sec.stodfoconvergence}. In Section \ref{sec:learningbounds} we discuss various noise scenarios and how sufficiently accurate models and estimates can be constructed in these cases. Finally, we present computational experiments based on these various noise scenarios
  in Section \ref{sec:computations}.
 
 \paragraph{Notations.}
Let $\|\cdot\|$ denote the  Euclidean norm and $B(x, \Delta)$ denote the ball of radius $\Delta$ around $x$, i.e., $B(x, \Delta): \{y:\| x-y\|\le \Delta\}$. 
 $\Omega$ denotes the probability sample space, according to the context, and a sample from that space is denoted by $\omega\in\Omega$. 
 As a rule, when we describe a random process within the algorithmic framework, uppercase letters, e.g. the $k$-th iterate $X_k$, will denote random variables, while lowercase letters will denote realizations of the random variable, e.g. $x_k=X_k(\omega)$ is the $k$-th iterate for a particular realization of our algorithm. 
 
We also list here, for convenience, several constants that are used in the paper to bound various quantities. These constants are denoted by $\kappa$ with subscripts indicating quantities that they are meant to bound. 
\begin{equation*}
    \begin{aligned}
\kappa_{ef} & \mbox{\quad``error in the function value",}\\
\kappa_{eg} &  \mbox{\quad``error in the gradient",}\\
\kappa_{Eef} & \mbox{\quad``expectation of the error in the function value",}\\
\kappa_{fcd} & \mbox{\quad``fraction of Cauchy decrease",}\\
\kappa_{bhm}&\mbox{\quad``bound on the Hessian of the models",}\\
\kappa_{et} & \mbox{\quad``error in Taylor expansion".}
     \end{aligned}
\end{equation*}

\section{Trust Region Method}\label{sec:TR-method}

We consider the trust-region class of methods for minimization of unconstrained functions. They operate as follows:
 at each iteration $k$, given the current iterate $x_k$ and a trust-region radius $\d_k$,
 a (random) model $m_k(x)$ is built, which serves as an approximation of $f(x)$ in $B(x_k, \d_k)$. 
 The model is assumed to be of the form
 \begin{equation}\label{eq:model}
 m_k(x_k+s)=f_k+g_k^\top s+s^\top H_k s.
 \end{equation}
 It is possible to generalize our framework to other forms of models, as long as all conditions on the models, described below, hold. We consider quadratic models for simplicity of the presentation and because they are the most common. 
 The model $m_k(x)$ is minimized (approximately) in $B(x_k, \d_k)$ to produce 
 a step $s_k$ and (random) estimates of $f(x_k)$ and $f(x_k+s_k)$ are obtained, denoted by  $f_k^0$ and $ f_k^s$ respectively. The achieved reduction is measured by comparing $f_k^0$ and $f_k^s$ and if reduction is deemed sufficient, then $x_k+s_k$ is chosen as the next iterate $x_{k+1}$. Otherwise the iterate remains  $x_k$. The trust-region radius $\d_{k+1}$ is then chosen by either increasing or decreasing $\d_k$ according to the outcome of the iteration. The details of the algorithm are presented in Algorithm \ref{algo:stodfosimple}. 

\balgorithm
\caption{{\sc Stochastic DFO with Random Models}}
\label{algo:stodfosimple}
 \textbf{(Initialization):} Choose an initial point $x_0$ and an initial trust-region radius $\d_0\in (0,\d_{\max})$ with $\d_{\max}>0$. Choose  constants $\gamma>1$, $\eta_1\in(0,1)$, $\eta_2 > 0$, set $k \gets 0$.\\
\label{step.model} \textbf{(Model construction):} Build a model $m_k(x_k+s)=f_k+g_k^\top s+s^\top H_k s$ that approximates $f(x)$ on $B(x_k, \d_k)$ with $s=x-x_k$.\\
\textbf{(Step calculation):} Compute $s_k = \arg \underset{s: \| s \|\le \d_k}{\min}  m_k(s) $ (approximately) so that it satisfies condition \eqref{eqn:CS}.\\ 
\textbf{(Estimates calculation):} Obtain estimates $f_k^0$ and $f_k^s $ of $f(x_k)$ and $f(x_k+s_k)$, respectively.\\
\textbf{(Acceptance of the trial point):} Compute $\rho_k = \dfrac{f_k^0-f_k^s}{m_k(x_k)-m_k(x_k+s_k)}.$ If $\rho_k\ge \eta_1$ and $ \| g_k\| \ge \eta_2 \d_k $, set $x_{k+1} =x_k+s_k$; otherwise, set $x_{k+1}=x_k$.\\
\textbf{(Trust-region  radius update):} If $\rho_k\ge \eta_1$ and $ \| g_k\| \ge \eta_2 \d_k $, set $\d_{k+1} =\min\{  \gamma \d_k,\d_{\max}\}$; otherwise $\d_{k+1}=\gamma^{-1} \d_k$; $k \gets k+1$ and go to step \ref{step.model}.\\
\ealgorithm

The trial step computed on each iteration has to provide sufficient decrease of the model; in other words it has to satisfy the following standard fraction of Cauchy decrease condition:
\begin{assumption}
\label{asmpCS}
For every $k$, the  step $s_k$ is computed so that
\begin{eqnarray}
m_k(x_k)-m_k(x_k+s_k) \ge \dfrac{\kappa_{fcd}}{2} \| g_k \| \min \left\{ \dfrac{\|g_k\|}{\|H_k\|},\d_k  \right\}  \label{eqn:CS}
\end{eqnarray}
for some constant $\kappa_{fcd}\in(0,1].$
\end{assumption}

If progress is achieved and a new iterate is accepted in the $k$-th iteration then we call this a \emph{successful iteration}. Otherwise, the iteration is unsuccessful (and no step is taken). Hence a successful iteration occurs when $\rho_k\ge\eta_1$ 
and $ \| g_k\| \ge \eta_2 \d_k $.  However, a successful iteration does not necessarily yield an actual reduction in the true function $f$. This is because the values of $f(x)$ are not accessible in our stochastic setting 
and the step acceptance decision is made merely based on the estimates of $f(x_k)$ and $f(x_k+s_k)$. If these estimates, $f_k^0$ and $f_k^s$,  are not accurate enough, a successful iteration can result in an increase of  the true function value.  Hence we consider two types of successful iterations - those where $f(x)$ is in fact decreased proportionally to  $f_k^0-f_k^s$, which  we  call {\em true} successful iterations, and all other successful iterations, where the decrease of $f(x)$ can be arbitrarily small or even negative, which we call {\em false} successful iterations. Our setting and algorithmic framework does not allow us to determine which successful iterations are true and which ones are false, however, we will be able to show that true successful iterations occur sufficiently often for convergence to hold, if the random estimates $f_k^0$ and $f_k^s$ are sufficiently accurate. 

A trust region framework based on random models was introduced and analyzed in \cite{Bandeiraetal2013}.  In that paper, the authors introduced the concept of probabilistically fully-linear models to determine the conditions that random models should satisfy for convergence of the algorithm to hold. However,  the randomness in the models  in their setting arises from the the construction process, and not from the noisy objective function. 
It  is assumed in \cite{Bandeiraetal2013} that the function values at the current iterate and the trial point can be computed exactly and hence all successful iterations are true in that case. 
In our case,  it is necessary to define a measure for the accuracy of the estimates $f_k^0$ and $f_k^s$ (which, as we will see, generally has to be tighter than the measure of  accuracy of the model).   We will use a modified version of the probabilistic estimates introduced in \cite{larson2013stochastic}.

\section{Probabilistic Models and Estimates} \label{sec:models-estimates}
The models in this paper are functions which are constructed on each iteration, based on some random samples of stochastic function $\tilde f(x)$. Hence, the models themselves are random and so is their behavior and influence on the iterations. Hence, $M_k$  will denote a random model in the $k$-th iteration, while we will use the notation $m_k=M_k(\omega)$ for its realizations. As a consequence of using  random models, the iterates $X_k$, the trust-region radii $\Delta_k$ and the steps $S_k$  are also random quantities, and so $x_k=X_k(\omega)$, $\d_k = \Delta_k(\omega)$, $s_k=S_k(\omega)$ will denote their respective realizations.  Similarly, let random quantities  $\{F_k^0,F_k^s\}$ denote the estimates of $f(X_k)$ and $f(X_k+S_k)$, with their realizations denoted by $f_k^0=F_k^0(\omega)$ and $f_k^s=F_k^s(\omega)$.  In other words, Algorithm \ref{algo:stodfosimple} results in a stochastic process $\{M_k,X_k,S_k, \Delta_k, F_k^0, F_k^s\}$. Our goal is to show that under certain conditions on the sequences $\{M_k\}$
and $\{F_k^0, F_k^s\}$ the resulting stochastic process has desirable convergence properties with probability one. In particular, we will 
assume that  models $M_k$
and estimates  $F_k^0, F_k^s$ are sufficiently accurate with sufficiently high probability, conditioned on the past. 

To formalize conditioning on the past, let $\mathcal{F}_{k-1}^{M\cdot F}$ denote the $\sigma$-algebra generated by $M_0,\cdots, M_{k-1}$ and $F_0,\cdots,F_{k-1}$ and let $\mathcal{F}_{k-{1}/{2}}^{M\cdot F}$ denote the $\sigma$-algebra generated by $M_0,\cdots, M_{k}$ and $F_0,\cdots,F_{k-1}$. 

 To formalize sufficient accuracy, let us  recall a measure for the accuracy of deterministic models introduced in \cite{DFObook} and \cite{ConnScheVice09} (with the exact notation introduced in \cite{billups2013derivative}). 

\begin{definition}
Suppose $\nabla f$ is Lipschitz continuous. A function $m_k$ is a \emph{$\kappa$-fully linear model of $f$ on $B(x_k,\d_k)$} provided, for $ \kappa = (\kappa_{ef}, \kappa_{eg})$ and  $\forall y \in B$,
\begin{eqnarray}\label{eq:fl-model}
\| \nabla f(y) - \nabla m_k(y)  \| & \le & \kappa_{eg} \d_k, \mbox{ and }\\
| f(y)-m_k(y)| &\le & \kappa_{ef} \d_k^2.\nonumber
\end{eqnarray}
\end{definition}

In this paper we extend  the following concept of probabilistically fully-linear models which is proposed in \cite{Bandeiraetal2013}.

\begin{definition}\label{def:pfl-models-old}
A sequence of random models $\{ M_k \}$ is said to be $\alpha$-probabilistically $\kappa$-fully linear with respect to the corresponding sequence $\{ B(X_k,\Delta_k)\}$ if the events 
\begin{equation}\label{eq:Ik}
I_k = \{ M_k \mbox{ is a } \kappa \mbox{-fully linear model of } f \mbox{ on }B(X_k,\Delta_k)  \}
\end{equation}
satisfy the condition
$$ P (I_k | \mathcal{F}_{k-1}^M)\ge \alpha ,$$
where $\mathcal{F}^M_{k-1}$ is the $\sigma$-algebra generated by $M_0,\cdots, M_{k-1}$.
\label{probabilisticmodel}
\end{definition}
These probabilistically fully-linear models have the very simple properties that they are fully-linear (i.e., accurate enough) with sufficiently high probability, conditioned on the past, and they can be arbitrarily inaccurate otherwise. 
This property is somewhat different from the properties of models typical to stochastic optimization (such as, for example, stochastic gradient based models), where assumptions on the expected value and the variance of the models is imposed. We will discuss this in more detail in Section \ref{sec:learningbounds}.

In this paper, aside from sufficiently accurate models, we require estimates of the function values $f(x_k)$, $f(x_k+s_k)$ that are sufficiently accurate.  This is needed in order to evaluate whether a step is successful, unlike the case in \cite{Bandeiraetal2013} where the exact values $f(x_k)$ and $f(x_k+s_k)$ are assumed to be available.   The following definition of accurate estimates
is  a modified version of that used in \cite{larson2013stochastic}.
 
\begin{definition}\label{def:pa-estimates-old}   
The estimates $f_k^0$ and $f_k^s$ are said to be $\epsilon_F$-accurate estimates of $f(x_k)$ and $f(x_k+s_k)$, respectively,  for a given $\delta_k$ if
\begin{equation}\label{eq:estimates}
 |f_k^0 - f(x_k) |  \le \epsilon_F \d_k^2 \mbox{ and } |f_k^s - f(x_k+s_k) | \le \epsilon_F \d_k^2.
 \end{equation}
\end{definition} 

We  now modify Definitions \ref{def:pfl-models-old} and \ref{def:pa-estimates-old}   
and introduce definitions  of  probabilistically accurate models and estimates which we will use throughout the remainder of the paper.

\begin{definition}\label{def:pfl-models}
A sequence of random models $\{ M_k \}$ is said to be $\alpha$-probabilistically $\kappa$-fully linear with respect to the corresponding sequence $\{ B(X_k,\Delta_k)\}$ if the events 
\begin{equation}\label{eq:Ik}
I_k = \{ M_k \mbox{ is a } \kappa \mbox{-fully linear model of } f \mbox{ on }B(X_k,\Delta_k)  \}
\end{equation}
satisfy the condition
$$ P (I_k | \mathcal{F}_{k-1}^{M\cdot F})\ge \alpha ,$$
where $\mathcal{F}^{M\cdot F}_{k-1}$ is the $\sigma$-algebra generated by $M_0,\cdots, M_{k-1}$ and $F_0,\cdots,F_{k-1}$.
\label{probabilisticmodel}
\end{definition}

\begin{definition}\label{def:pa-estimates} A sequence of random estimates $\{F_k^0,F_k^s\}$ is said to be $\beta$-probabilistically $\epsilon_F$-accurate with respect to the corresponding sequence $\{X_k,\Delta_k,S_k\}$ if the events
\begin{equation}\label{eq:Jk}
 J_k = \{F_k^0,F_k^s \mbox{ are } \epsilon_F\mbox{-accurate estimates of }f(x_k) \mbox{ and }f(x_k+s_k), \mbox{ respectively, \ for\ } \Delta_k  \}  
 \end{equation}
satisfy the condition 
$$P( J_k|\mathcal{F}_{k-1/2}^{M\cdot F}  )\ge \beta,$$
where $\epsilon_F$ is a fixed constant and $\mathcal{F}_{k-{1}/{2}}^{M\cdot F}$ is the $\sigma$-algebra generated by $M_0,\cdots, M_{k}$ and $F_0,\cdots,F_{k-1}$. 
\label{asmpesti}
\end{definition}


Using Definitions \ref{def:pfl-models} and \ref{def:pa-estimates} we  will  require in our analysis that our method has access to 
$\alpha$-probabilistically $\kappa$-fully linear models, for some fixed $ \kappa = (\kappa_{ef}, \kappa_{eg})$ and to 
$\beta$-probabilistically $\epsilon_F$-accurate estimates, for some fixed, sufficiently small $\epsilon_F$. 
Thus, the model and the estimate accuracy will be assumed to be  proportional to $\delta_k^2$ (with some probability), the condition on the estimates is somewhat tighter because of the  upper bound  on $\epsilon_F$. 
 However, we will see that this upper bound is not very small. 

Procedures for obtaining probabilistically fully-linear models and probabilistically accurate estimates under different models of noise are discussed in Section \ref{sec:learningbounds}. 
\section{Convergence Analysis}\label{sec.stodfoconvergence}

We now present first-order convergence analysis for  the general framework described in Algorithm \ref{algo:stodfosimple}.  Towards that end, we assume that the function $f$ and its gradient are Lipschitz continuous in regions considered by the algorithm realizations. 

\begin{assumption} [{\rm Assumptions on $f$}] Let $x_0$ and $\d_{\max}$ be given. 
Let ${\cal L}(x_0)$ define the set in $\mathbb{R}^n$ which contains all iterates of our algorithm. Assume that $f$ is bounded from below on ${\cal L}(x_0)$.
Assume also that the function $f$ and its gradient $\nabla f$ are $L$-Lipschitz continuous on the set ${\cal L}_{enl}(x_0)$, where ${\cal L}_{enl}(x_0)$ defines the region considered by the algorithm realizations 
$${\cal L}_{enl}(x_0)=\underset{x\in L(x_0)}{\bigcup} B(x; \d_{\max}). $$\label{asmpf}
\end{assumption}

\begin{remark}
In the case of deterministic functions  ${\cal L}(x_0)=\{ x\in \mathbb{R}^n:\, f(x)\le f(x_0)  \}$, because algorithm iterates never increase the objective function value, hence they do not step outside the initial level set. However, here we allow iterates to increase the function value, because the true function value is not known. Such iterates, as we will see, may happen with some (relatively small) probability, hence the algorithm can venture outside the initial level set. 
Hence we choose to make the assumption above, which of course depends on the algorithmic behavior. Clearly, if we assume a global Lipschitz constant and global lower bound, then the above assumption always holds. If we prefer to weaken this assumption, then there are several algorithmic remedies
possible, however, they will make our analysis more complicated and we choose to leave it for future work. It will be evident from our analysis, that 
with a probability arbitrarily close to 1 we can bound the set ${\cal L}(x_0)$. 
\end{remark}


The second assumption provides a uniform upper bound on the model Hessian.
\begin{assumption}
\label{asmpH}There exists a positive constant $\kappa_{bhm} $ such that, for every $k$, the Hessian $H_k$ of all realizations $m_k$ of $M_k$ satisfy
\begin{eqnarray*}
\| H_k \|\le \kappa_{bhm} .
\end{eqnarray*}\end{assumption}
Note that since we are concerned with convergence to a first order stationary point in this paper, the bound $\kappa_{bhm}$ can be chosen to be any nonnegative number, including zero. Allowing a larger bound will give more flexibility to the algorithm and may allow better Hessian approximations, but as we will see in the convergence analysis, this imposes restrictions on the trust region radius and some other algorithmic parameters. 

We now state the following result from martingale literature \cite{durrett2010probability} (see Exercise 5.3.1) that will be useful later in our analysis. 
\begin{theorem}\label{submartingale} 
Let $G_k$ be a submartingale, i.e., a sequence of random variables which, for every $k$, 
$$E[G_k|\mathcal{F}_{k-1}^G]\ge G_{k-1},$$
where $\mathcal{F}_{k-1}^G=\sigma(G_0,\ldots,G_{k-1})$ is the $\sigma$-algebra generated by $G_0,\ldots, G_{k-1}$, and $E[G_k|\mathcal{F}_{k-1}^G]$ denotes the conditional expectation of $G_k$ given the past history of events $\mathcal{F}_{k-1}^G$.

Assume further that $G_k-G_{k-1}\le M<\infty$, for every $k$. Then,
\begin{eqnarray}
P\left( \left\{ \underset{k\to \infty}{\lim} G_k <\infty  \right\} \cup\left\{  \underset{k\to \infty}{\lim\sup}\,  G_k=\infty  \right\}  \right)=1.
\end{eqnarray}
\end{theorem}

%

We now prove some auxiliary lemmas that provide conditions under which decrease of the true objective function $f(x)$ is guaranteed. The first lemma states that if the trust region radius is small enough relative to the size of the model gradient and if the model
is fully linear, then the step $s_k$ provides a decrease in $f(x)$ proportional to the size of the model gradient. Note that the trial step may still be rejected if the estimates $f_k^0$ and $f_k^s$ are not accurate enough. 

\begin{lemma}\label{lemma.delta.1}
Suppose that a model $m_k$ of the form \eqref{eq:model} is a $(\kappa_{ef},\kappa_{eg})$-fully linear model of $f$ on $B(x_k,\d_k)$. If
$$\d_k\le\min\left\{ \frac{1}{\kappa_{bhm}},\frac{\kappa_{fcd}}{8\kappa_{ef}}  \right\} \|g_k\|,$$
then the trial step $s_k$ leads to an improvement in $f(x_k+s_k)$ such that
\begin{eqnarray}f(x_k+s_k)-f(x_k)\le - \frac{\kappa_{fcd}}{4} \|g_k\|\d_k .
\end{eqnarray}
\end{lemma}
\begin{proof} 

Using the Cauchy decrease condition, the upper bound on model Hessian and the fact that $\|g_k\|\ge \kappa_{bhm}\d_k$, we have
\begin{eqnarray*}
m_k(x_k)-m_k(x_k+s_k) \ge\frac{\kappa_{fcd}}{2} \| g_k \| \min \left\{ \frac{\| g_k\| }{\| H_k \|} ,\delta_k\right\}= \frac{\kappa_{fcd}}{2}\|g_k\|\delta_k.
\end{eqnarray*}

Since the model is $\kappa$-fully linear, one can express the improvement in $f$ achieved by $s_k$ as
\begin{eqnarray*}
&&f(x_k+s_k)-f(x_k)\\
&=& f(x_k+s_k)-m(x_k+s_k)+m(x_k+s_k)-m(x_k)+m(x_k)-f(x_k)\\
&\le &2\kappa_{ef}\d_k^2-\frac{\kappa_{fcd}}{2}\|g_k\|\delta_k\\
&\le & - \frac{\kappa_{fcd}}{4} \|g_k\|\d_k,
\end{eqnarray*}
where the last inequality is implied by $\d_k\le\frac{\kappa_{fcd}}{8\kappa_{ef}} \|g_k\|$.
\hfill\end{proof} 

The next lemma shows that for $\delta_k$ small enough relative to the size of the true gradient $\nabla f(x_k)$, the guaranteed decrease in the objective function, provided by $s_k$, is proportional to the size of the true gradient. 

\begin{lemma}\label{lemma.delta.2} Under Assumption   \ref{asmpH}, suppose that a model is $(\kappa_{ef},\kappa_{eg})$-fully linear on $B(x_k,\d_k)$. If
\begin{eqnarray}\label{condition.lemma.delta.2}
\d_k\le \min \left\{ \frac{1}{\kappa_{bhm}+\kappa_{eg}}, \frac{1}{\frac{8\kappa_{ef}}{\kappa_{fcd}}+\kappa_{eg}}\right\} \|\nabla f(x_k)\|,
\end{eqnarray}
then the trial step $s_k$ leads to an improvement in $f(x_k+s_k)$ such that
\begin{eqnarray}f(x_k+s_k)-f(x_k)\le - C_1 \|\nabla f(x_k)\|\d_k , \label{eqn.true.decrease}
\end{eqnarray}
where $C_1=\frac{\kappa_{fcd}}{4}\cdot \max\left\{ \frac{\kappa_{bhm}}{\kappa_{bhm}+\kappa_{eg}},\frac{8\kappa_{ef}}{8\kappa_{ef}+\kappa_{fcd}\kappa_{eg}}\right\}.$
\end{lemma}

\begin{proof}
The definition of a $\kappa$-fully-linear model yields that
$$\|g_k\|\ge \|\nabla f(x)\|-\kappa_{eg}\delta_k .$$ 
Since condition (\ref{condition.lemma.delta.2})  implies that $\| \nabla f(x_k) \| \ge\max\left\{ \kappa_{bhm}+\kappa_{eg},\frac{8\kappa_{ef}}{\kappa_{fcd}}+\kappa_{eg}\right\}\d_k   $, we have
$$\|  g_k \| \ge \max\left\{ \kappa_{bhm}  , \frac{8\kappa_{ef}}{\kappa_{fcd}}\right\} \d_k.$$
Hence, the conditions of Lemma \ref{lemma.delta.1} hold and we have 
\begin{eqnarray}\label{lemma.delta.2.eqn.1}
f(x_k+s_k)-f(x_k)\le  - \frac{\kappa_{fcd}}{4} \|g_k\|\d_k .
\end{eqnarray}
Since $\|g_k\|\ge \|\nabla f(x)\|-\kappa_{eg}\delta_k$ in which $\d_k$ satisfies (\ref{condition.lemma.delta.2}), we also have 
\begin{eqnarray}\label{lemma.delta.2.eqn.2}
\|g_k\| \ge   \max\left\{ \frac{\kappa_{bhm}}{\kappa_{bhm}+\kappa_{eg}},\frac{8\kappa_{ef}}{8\kappa_{ef}+\kappa_{fcd}\kappa_{eg}}\right\}\|\nabla f(x_k)\| .
\end{eqnarray}
Combining (\ref{lemma.delta.2.eqn.1}) and (\ref{lemma.delta.2.eqn.2}) yields (\ref{eqn.true.decrease}).
\hfill\end{proof}

We now prove the lemma that states that, if a) the estimates are sufficiently accurate, b) the model is fully-linear and c) the trust-region radius is sufficiently small relatively to the size of the model gradient, then a successful step is guaranteed. 

\begin{lemma}\label{lemma.delta.3} Under Assumption \ref{asmpH}, suppose that $m_k$ is $(\kappa_{ef},\kappa_{eg})$-fully linear on $B(x_k,\d_k)$ and the estimates $\{ f_k^0,f_k^s \}$ are $\epsilon_F$-accurate with $\epsilon_F\le \kappa_{ef}$. If 
\begin{eqnarray}\label{condition.lemma.delta.3}
 \d_k \le \min \left\{ \dfrac{1}{\kappa_{bhm}} , \frac{1}{\eta_2}, \dfrac{ \kappa_{fcd}  (1-\eta_1)}{ 8\kappa_{ef} }   \right\} \| g_k \| , 
 \end{eqnarray}
then the $k$-th iteration is successful.
\end{lemma}

\begin{proof}
Since $  \d_k \le \frac{\| g_k \|}{\kappa_{bhm}}$, the Cauchy decrease condition and the uniform bound on $H_k$ immediately yield that 
\begin{eqnarray}
m_k(x_k) -m_k(x_k+s_k) \ge \dfrac{\kappa_{fcd}}{2} \|g_k\| \min\left\{  \dfrac{\| g_k\|}{\kappa_{bhm}  } ,\d_k\right\} = \dfrac{\kappa_{fcd}}{2}\| g_k \|\d_k.\label{eqn:deltainc1}
\end{eqnarray}
The model $m_k$ being  $(\kappa_{ef},\kappa_{eg})$-fully linear implies that
\begin{eqnarray}
|f(x_k)-m_k(x_k)|&\le& \kappa_{ef}\d_{k}^2,\mbox{ and } \label{eqn:deltainc2}\\
|f(x_k+s_k)-m_k(x_k+s_k)|&\le& \kappa_{ef}\d_{k}^2.\label{eqn:deltainc3}
\end{eqnarray}
Since the estimates are $\e_F$-accurate with $\epsilon_F\le \kappa_{ef}$, we obtain\begin{eqnarray}
|f_k^0-f(x_k)|\le \kappa_{ef}\d_k^2,\mbox{ and } |f_k^s-f(x_k+s_k)|\le \kappa_{ef}\d_k^2.\label{eqn:deltainc4}
\end{eqnarray}

We have
\begin{eqnarray*}
\rho_k &= &\dfrac{f_k^0-f_k^s}{m_k(x_k)-m_k(x_k+s_k)}\\
&=&  \dfrac{f_k^0-f(x_k) }{m_k(x_k)-m_k(x_k+s_k)}+\dfrac{f(x_k) -m_k(x_k)}{m_k(x_k)-m_k(x_k+s_k)}+\dfrac{m_k(x_k)-m_k(x_k+s_k) }{m_k(x_k)-m_k(x_k+s_k)}\\
&& + \dfrac{m_k(x_k+s_k)-f(x_k+s_k)}{m_k(x_k)-m_k(x_k+s_k)}+\dfrac{f(x_k+s_k)   -f_k^s   }{m_k(x_k)-m_k(x_k+s_k)},
\end{eqnarray*}
which , combined with (\ref{eqn:deltainc1})-(\ref{eqn:deltainc4}), implies
\begin{eqnarray*}
| \rho_k-1|\le \dfrac{8\kappa_{ef}\d_k^2}{\kappa_{fcd}\|g_k\| \d_k }\le 1-\eta_1,
\end{eqnarray*}
where we have used the assumption $\d_k \le \frac{ \kappa_{fcd}(1-\eta_1) }{ 8\kappa_{ef}  }  \| g_k \|   $ to deduce the last inequality. Hence, $\rho_k\ge\eta_1$. Moreover, since $\|  g_k\|\ge \eta_2\d_k$, the $k$-th iteration is successful.
\hfill\end{proof} 

Finally, we state and prove the lemma which guarantees an amount of decrease of the  objective function on a true successful iteration.

\begin{lemma}\label{lemma.delta.4}
Under Assumption \ref{asmpH}, suppose that the estimates $\{f_k^0,f_k^s  \}$ are $\e_F$-accurate with $\e_F< \frac{1}{4}\eta_1\eta_2 \kappa_{fcd}\min \left\{ \frac{\eta_2}{\kappa_{bhm}},1 \right\}$. If a trial step $s_k$ is accepted (a successful iteration occurs), then the improvement in $f$ is bounded below as follows
\begin{eqnarray}\label{eqn.lemma.4}
f(x_{k+1})-f(x_k)\le -C_2\delta_k^2,\end{eqnarray}
where $C_2 =\frac{1}{2}\eta_1\eta_2 \kappa_{fcd}\min \left\{ \frac{\eta_2}{\kappa_{bhm}},1 \right\}-2\epsilon_F>0$.
\end{lemma}
\begin{proof} An iteration being successful indicates that $\|g_k\|\ge\eta_2 \d_k$ and $\rho\ge \eta_1$. Thus,
\begin{eqnarray*}
f_k^0-f_k^s &\ge&\eta_1(m_k(x_k)-m_k(x_k+s_k))\\
&\ge&\eta_1 \frac{\kappa_{fcd}}{2} \| g_k \| \min \left\{ \frac{\| g_k\| }{\| H_k \|} ,\delta_k\right\}\\
&\ge& \frac{1}{2} \eta_1\eta_2 \kappa_{fcd}\min \left\{ \frac{\eta_2}{\kappa_{bhm}},1 \right\}\delta_k^2.
\end{eqnarray*}
Then, since the estimates are $\epsilon_F$-accurate, we have that the improvement in $f$ can be bounded as
\begin{eqnarray*}
f(x_k+s_k)-f(x_k)= f(x_k+s_k)-f_k^s+f_k^s-f_k^0+f_k^0-f(x_k)
\le  -C_2\delta_k^2,
\end{eqnarray*}
where $C_2 =\frac{1}{2}\eta_1\eta_2 \kappa_{fcd}\min \left\{ \frac{\eta_2}{\kappa_{bhm}},1 \right\}-2\epsilon_F>0$.
\hfill\end{proof}

To prove convergence of Algorithm \ref{algo:stodfosimple}  we will need to assume that   models $\{M_k\}$ and estimates $\{F_k^0, F_k^s\}$ are sufficiently accurate with sufficiently high probability.

\begin{assumption}\label{ass:modelsandestimates} Given values of $\alpha,\beta\in (0,1)$ and $\epsilon_F>0$, there exist $\kappa_{eg}$ and $\kappa_{ef}$ such that the 
the sequence of models $\{M_k\}$ and estimates $\{F_k^0, F_k^s\}$ generated by Algorithm \ref{algo:stodfosimple} are, respectively, $\alpha$-probabilistically 
$(\kappa_{ef}, \kappa_{eg})$- fully-linear and $\beta$-probabilistically $\epsilon_F$-accurate. 
\end{assumption}

\begin{remark} Note that this assumption is a statement about the existence of constants $\kappa=(\kappa_{ef},\kappa_{eg})$ given an $\alpha$, $\beta$ and $\epsilon_F$ - we will determine exact conditions on $\alpha$, $\beta$ and $\epsilon_F$ in Theorem \ref{lemma.delta.gotozero} and Lemma \ref{th:constants} below.
 \end{remark}

The following theorem states that the trust-region radius converges to zero with probability $1$.

\begin{theorem}\label{lemma.delta.gotozero}
Let Assumptions  \ref{asmpf} and \ref{asmpH}  be satisfied and assume that  in Algorithm \ref{algo:stodfosimple} the following holds. 
\begin{itemize}
\item The  step acceptance parameter $\eta_2$ is chosen so that  
\begin{eqnarray}
\eta_2&\geq&\max \left\{ \kappa_{bhm}, \frac{8\kappa_{ef}}{\kappa_{fcd}(1-\eta_1)}  \right\}.\label{eq:cond_eta2}\end{eqnarray}
\item
The accuracy parameter of the estimates satisfies
\begin{eqnarray}
\epsilon_F&\leq&\min \left\{ \kappa_{ef}, \frac{1}{8}\eta_1\eta_2 \kappa_{fcd} \right\}. \label{eq:cond_epsF}
\end{eqnarray}
\end{itemize}

Then $\alpha$ and $\beta$ can be chosen so that, if 
Assumption \ref{ass:modelsandestimates} holds for these values, then the sequence of trust-region radii, $\{\Delta_k\}$, generated by Algorithm \ref{algo:stodfosimple} satisfies 
\begin{equation}\label{eq:sumdeltasquares}
\sum_{k=0}^\infty \Delta^2_k<\infty
\end{equation}
almost surely.
\end{theorem}

\begin{proof}
We base our proof on properties of the  random function $\Phi_k =\nu f(X_{k})+(1-\nu)\Delta_{k}^2 $, where $\nu \in (0,1)$ is a fixed constant, which is specified below. A  similar function is used in the analysis   in  \cite{larson2013stochastic}, but the analysis itself is different. The overall goal is to show that there exists a constant $\sigma>0$ such that for all $k$
\begin{equation}\label{eq:expectationphi}
E[\Phi_{k+1}-\Phi_{k}|\mathcal{F}_{k-1}^{M\cdot F} ]\le -\sigma \Delta_k^2< 0. 
\end{equation}
Since $f$ is bounded from below and $\Delta_k>0$, we have that $\Phi_{k}$ is bounded from below for all $k$ and hence if \eqref{eq:expectationphi} holds on every iteration, then by summing \eqref{eq:expectationphi} over $k\in(1,\infty)$ and taking expectations on both sides we can conclude that  \eqref{eq:sumdeltasquares} holds with probability $1$. 
Hence, to prove the theorem we need to show that  \eqref{eq:expectationphi} holds on each iteration. 
\vskip0.1cm

Let us pick some  constant  $\zeta$ which satisfies
\begin{equation}\label{eq:zeta}
\zeta\ge \kappa_{eg}+\max\left\{  \eta_2,  \frac{8\kappa_{ef}}{\kappa_{fcd}(1-\eta_1)}\right\}.
\end{equation}
We now consider two possible cases:  $\| \nabla f(x_k)  \| \ge \zeta \delta_k $ and $\| \nabla f(x_k)  \| < \zeta \delta_k$.
We will show that \eqref{eq:expectationphi}
holds in both cases and hence it holds on every iteration. Given $\zeta$ we now select $\nu\in (0,1)$ such that
\begin{eqnarray}\label{eq:nu}
\frac{\nu}{1-\nu}&>& \max \left\{ \frac{4\gamma^2}{\zeta C_1} , \frac{4\gamma^2}{ \eta_1\eta_2\kappa_{fcd}}, \frac{\gamma^2}{\kappa_{ef}} \right\},
\end{eqnarray}
with $C_1$ defined as in Lemma \ref{lemma.delta.2}.

As usual, let $x_k$, $\delta_k$, $s_k$, $g_k$, and $\phi_k$ denote realizations of random quantities $X_k$, $\Delta_k$, $S_k$, $G_k$, and  $\Phi_k$, respectively.

Let us consider some realization of Algorithm  \ref{algo:stodfosimple}. Note that on all successful iterations, $x_{k+1}=x_k+s_k$ and $\delta_{k+1}=\min\{\gamma \delta_k, \delta_{max}\}$ with $\gamma>1$, hence
\begin{eqnarray} 
\phi_{k+1}-\phi_k\leq\nu (f(x_{k+1})-f(x_k))+(1-\nu)(\gamma^2-1)\delta_k^2.\label{eqn.suc.phi}
\end{eqnarray} 
On  all unsuccessful iterations, $x_{k+1}=x_k$ and $\delta_{k+1}=\frac{1}{\gamma} \delta_k$, i.e. 
\begin{eqnarray}
\phi_{k+1}-\phi_k=(1-\nu)(\frac{1}{\gamma^2}-1)\delta_k^2\equiv b_1<0.\label{eqn.unsuc.phi}
\end{eqnarray}

For each iteration and each of the two cases we consider,  we will  analyze the four possible combined outcomes of the events $I_k$ and $J_k$ as defined in \eqref{eq:Ik} and \eqref{eq:Jk}, respectively. 

Before presenting the formal proof let us outline the key ideas. We will show that, unless both the model and the estimates are bad on iteration $k$, we select $\nu\in (0,1)$ sufficiently close to $1$, so that the
 decrease in $\phi_k$ on a successful iteration is greater than the decrease on an unsuccessful iteration (which is equal to $b_1$, according to \eqref{eqn.unsuc.phi}). When the model and the estimates are both bad,
 an increase in $\phi_k$ may occur. This increase is bounded by a value proportional to $\delta_k^2$ when $\|\nabla f(x_k)\| < \zeta\delta_k$. When $\| \nabla f(x_k)  \| \ge \zeta \delta_k$, though, the increase in $\phi_k$ may be proportional to 
 $\|\nabla f(x_k)\|\delta_k$. However, since iterations with good models and good estimates will provide \emph{decrease} in $\phi_k$ also proportional to $\|\nabla f(x_k)\|\delta_k$, by choosing values of $\alpha$ and $\beta$ close enough to 1, we can ensure that in expectation $\phi_k$ decreases. 
 
 We now present the proof. 
 
\paragraph{Case 1: $\| \nabla f(x_k)  \| \ge \zeta \delta_k $.} 

\begin{itemize}
\item[a.] $I_k$ and $J_k$ are both true, i.e., both the  model and the  estimates are good on iteration $k$.  
From the definition of $\zeta$, we know
\[
 \|\nabla f(x_k)\|\ge \left(\kappa_{eg}+ \max\left\{  \eta_2,  \frac{8\kappa_{ef}}{\kappa_{fcd}(1-\eta_1)}\right\}\right)\delta_k.
 \]
Then since the model $m_k$ is $\kappa$-fully linear and, from $\eta_2>\kappa_{bhm}$, $\e_F\le \kappa_{ef}$
 and $0<\eta_1<1$, it is easy to show that  the condition (\ref{condition.lemma.delta.2}) in Lemma \ref{lemma.delta.2} holds. Therefore, the trial step $s_k$ leads to a decrease in $f$ as in (\ref{eqn.true.decrease}).

Moreover, since 
$$\|g_k\|\ge \|\nabla f(x_k)\|-\kappa_{eg}\d_k\ge (\zeta-\kappa_{eg})\d_k\ge \max \left\{\eta_2, \frac{8\kappa_{ef}}{\kappa_{fcd}(1-\eta_1)}  \right\}\d_k,$$
 and the estimates $\{ f_k^0,f_k^s \}$ are $\e_F$-accurate, with $\e_F\le \kappa_{ef}$, the condition (\ref{condition.lemma.delta.3}) in Lemma \ref{lemma.delta.3} holds. Hence, iteration $k$ is successful, i.e. $x_{k+1}=x_k+s_k$ and $\d_{k+1}=\gamma\d_k$. 

Combining (\ref{eqn.true.decrease}) and (\ref{eqn.suc.phi}),  we get
\begin{eqnarray}
\phi_{k+1}-\phi_k\le -\nu C_1\| \nabla f(x_k)\|\delta_k +(1-\nu)(\gamma^2-1)\delta_k^2\equiv b_2,\label{eqn.suc.phi.truedecrease1}
\end{eqnarray}
with $C_1$ defined in Lemma \ref{lemma.delta.2}.  Since  $\|\nabla f(x_k)  \| \ge \zeta \delta_k$ we have 
\begin{eqnarray}
b_2\leq [-\nu C_1 \zeta+(1-\nu)(\gamma^2-1)]\delta_k^2<0,\label{eq:truedecneg}
\end{eqnarray}
for $\nu\in(0,1)$ satisfying \eqref{eq:nu}. 


\item[b.] $I_k$ is true and  $J_k$ is false, i.e., we have a good model and bad estimates  on iteration  $k$.

In this case, Lemma \ref{lemma.delta.2} still holds, that is $s_k$ yields a sufficient decrease in $f$, hence, if the iteration is successful, we obtain \eqref{eqn.suc.phi.truedecrease1} and \eqref{eq:truedecneg}. However, the step can be erroneously  rejected, because of inaccurate probabilistic estimates, in which case  we have an unsuccessful iteration and (\ref{eqn.unsuc.phi}) holds. Since  \eqref{eq:nu} holds the right hand side of the first relation in (\ref{eq:truedecneg}) is strictly smaller than the right hand side of the first relation in (\ref{eqn.unsuc.phi}) and therefore, \eqref{eqn.unsuc.phi} holds whether the iteration is successful or not. 

\item[c.] $I_k$ is false and  $J_k$ is true, i.e., we have a bad model and good estimates  on iteration  $k$. 
 In this case, iteration $k$ can be either successful or unsuccessful.  In the unsuccessful case (\ref{eqn.unsuc.phi}) holds.  When the iteration is successful, since the estimates are $\e_F$-accurate and \eqref{eq:cond_epsF} 
 holds  then by Lemma \ref{lemma.delta.4} (\ref{eqn.lemma.4}) holds with $C_2\ge \frac{1}{4}\eta_1\eta_2\kappa_{fcd}$. Hence, 
 in this case we have 
\begin{eqnarray}
\phi_{k+1}-\phi_k\le [-\nu C_2+(1-\nu)(\gamma^2-1)]\delta_k^2. \label{eqn.suc.phi.truedecrease2}
\end{eqnarray}

Again, due to  the choice of $\nu$  satisfying  \eqref{eq:nu}  we have that, as in case (b),  \eqref{eqn.unsuc.phi} holds whether the iteration is successful or not. 

\item[d.] $I_k$ and $J_k$ are both false, i.e., both the  model and the  estimates are bad on iteration $k$.

Inaccurate estimates can cause the algorithm to accept a bad step, which may lead to an increase both in $f$ and in $\delta_k$. Hence in this case $\phi_{k+1}-\phi_{k}$ may  be positive. However, combining the Taylor expansion of $f(x_k)$ at $x_k+s_k$ and the Lipschitz continuity of $\nabla f(x)$ we can bound the amount of increase  in  $f$, hence bounding $\phi_{k+1}-\phi_{k}$ from above. By adjusting the probability of outcome (d) to be sufficiently small, we can  ensure that in expectation $\Phi_k$ is sufficiently reduced.  

In particular,  from Taylor's Theorem and Lipschitz continuity of $\nabla f(x)$ we have, respectively, 
\begin{eqnarray*}
f(x_k)-f(x_k+s_k)&\ge&\nabla f(x_k+s_k)^T(-s_k)-\dfrac{1}{2}L\delta_k^2,\mbox{ and }\\
\| \nabla f(x_k+s_k)-\nabla f(x_k)   \|&\le& Ls_k\le L\delta_k.
\end{eqnarray*}
From this we can derive  that any  increase of $f(x_k)$   is bounded by
\begin{eqnarray*}
 f(x_k+s_k)-f(x_k)\le C_3 \|\nabla f(x_k)\|\delta_k,
 \end{eqnarray*}
where $C_3=1+\frac{3L}{2\zeta}$.
Hence, the change in function $\phi$ is bounded: 
\begin{eqnarray}
\phi_{k+1}-\phi_k\le \nu C_3 \|\nabla f(x_k)\| \delta_k+(1-\nu)(\gamma^2-1)\delta_k^2\equiv b_3.\label{eqn.suc.phi.falseincrease}
\end{eqnarray}
\end{itemize}
\vskip0.5cm

Now we are ready to take the expectation of $\Phi_{k+1}-\Phi_{k}$ for the case when $\|\nabla f(X_k)  \| \geq \zeta \Delta_k$. We know that case (a) occurs with a probability at least $\alpha\beta$ (conditioned on the past) and in that case $\phi_{k+1}-\phi_{k}=b_2<0$ with $b_2$ defined
in  \eqref{eqn.suc.phi.truedecrease1}. Case (d) occurs with probability at most $(1-\alpha)(1-\beta)$ and that case $\phi_{k+1}-\phi_{k}$ is bounded from above by  $b_3>0$. Cases (b) and (c) occur otherwise
and in those cases $\phi_{k+1}-\phi_{k}$ is bounded from above by $b_1<0$, with $b_1$ defined in (\ref{eqn.unsuc.phi}). Finally we note that $b_1>b_2$ due to our choice of $\nu$.

Hence,  we can combine (\ref{eqn.unsuc.phi}), (\ref{eqn.suc.phi.truedecrease1}), (\ref{eqn.suc.phi.truedecrease2}), and (\ref{eqn.suc.phi.falseincrease}), and use $B_1$, $B_2$, and $B_3$ as random counterparts of $b_1$, $b_2$, and $b_3$,  to obtain the following bound
\begin{eqnarray*}
&&E[\Phi_{k+1}-\Phi_{k}|\mathcal{F}_{k-1}^{M\cdot F},   \{\| \nabla f(X_k)  \| \geq \zeta \Delta_k\}] \\
&\le&\alpha\beta B_2 + [\alpha(1-\beta)+(1-\alpha)\beta]B_1 + (1-\alpha)(1-\beta)B_3\\
&=&\alpha\beta[-\nu C_1\| \nabla f(X_k)\|\Delta_k +(1-\nu)(\gamma^2-1)\Delta_k^2 ]\\
&&+[\alpha(1-\beta)+(1-\alpha)\beta](1-\nu)(\frac{1}{\gamma^2}-1)\Delta_k^2 \\
&&+ (1-\alpha)(1-\beta)[\nu C_3 \|\nabla f(X_k)\| \Delta_k+(1-\nu)(\gamma^2-1)\Delta_k^2].
\end{eqnarray*}
Rearranging the terms we obtain 
\begin{eqnarray*}
&&E[\Phi_{k+1}-\Phi_{k}|\mathcal{F}_{k-1}^{M\cdot F},  \{\| \nabla f(X_k)  \| \geq \zeta \Delta_k\}] \\
&\le&[-\nu C_1\alpha\beta + (1-\alpha)(1-\beta)\nu C_3] \| \nabla f(X_k)\|\Delta_k \\
&&+[\alpha\beta - \frac{1}{\gamma^2}(\alpha(1-\beta)+(1-\alpha)\beta) +(1-\alpha)(1-\beta)](1-\nu)(\gamma^2-1)\Delta_k^2\\
&\le& [- C_1\alpha\beta + (1-\alpha)(1-\beta) C_3] \nu \|\nabla f(X_k)\|\Delta_k+ (1-\nu)(\gamma^2-1)\Delta_k^2, 
\end{eqnarray*}
where the last inequality holds because $\alpha\beta - \frac{1}{\gamma^2}(\alpha(1-\beta)+(1-\alpha)\beta) +(1-\alpha)(1-\beta)\le[\alpha+(1-\alpha)][(\beta+(1-\beta)]=1.$

Let us choose $0<\alpha\leq 1$ and  $0<\beta\leq 1$ so that they satisfy 
\begin{eqnarray}\label{eq:alpha.beta.condition.1}
\frac{(\alpha\beta-\frac{1}{2})}{(1-\alpha)(1-\beta)}&\ge& \frac{ C_3 }{C_1}
\end{eqnarray}
which implies 
\begin{eqnarray*}
[C_1\alpha\beta - (1-\alpha)(1-\beta) C_3]\geq \frac{1}{2}C_1\geq 2\frac{(1-\nu)(\gamma^2-1)}{\nu \zeta},
\end{eqnarray*}
where the last inequality is the result of \eqref{eq:nu}. It is important to note that the quantity $\frac{1}{2}$ in the numerator of \eqref{eq:alpha.beta.condition.1}
is chosen so that the first inequality of the above expression holds. This quantity can be made arbitrarily small, and with appropriate adjustment to \eqref{eq:nu}
one can have 
\begin{eqnarray*}
[C_1\alpha\beta - (1-\alpha)(1-\beta) C_3]\geq \theta_1C_1\geq \theta_2\frac{(1-\nu)(\gamma^2-1)}{\nu \zeta},
\end{eqnarray*}
where $\theta_1$ is positive and arbitrarily close to zero and $\theta_2>1$ and arbitrarily close to one. We will return to this comment in a later remark, but for the sake of simplicity, we choose $\theta_1=\frac{1}{2}$  and $\theta_2=2$.

Recall that  $\|\nabla f(X_k)\|\geq \zeta \Delta_k$, hence 
\begin{eqnarray*}
& &[- C_1\alpha\beta + (1-\alpha)(1-\beta) C_3] \nu \|\nabla f(X_k)\|\Delta_k+ (1-\nu)(\gamma^2-1)\Delta_k^2\\
& \le& \frac{1}{2}[- C_1\alpha\beta + (1-\alpha)(1-\beta) C_3] \nu \|\nabla f(X_k)\|\Delta_k\le  -\frac{1}{4} C_1 \nu \|\nabla f(X_k)\|\Delta_k
\end{eqnarray*}

In summary, we have 
\begin{equation}\label{eq:expectPhi1}
E[\Phi_{k+1}-\Phi_{k}|\mathcal{F}_{k-1}^{M\cdot F},  \{\| \nabla f(X_k)  \| \geq \zeta \Delta_k\}] \le - \frac{1}{4} C_1\nu \|\nabla f(X_k)\| \Delta_k
\end{equation}
and 
\begin{equation}\label{eq:final.decrease.in.phi.1}
E[\Phi_{k+1}-\Phi_{k}|\mathcal{F}_{k-1}^{M\cdot F},  \{\| \nabla f(X_k)  \| \geq \zeta \Delta_k\}] \le -\frac{1}{2}
(1-\nu)(\gamma^2-1)\Delta_k^2
\end{equation}
For the purposes of this lemma and the liminf-type convergence result, which will follow,  bound \eqref{eq:final.decrease.in.phi.1} is sufficient. We will use bound \eqref{eq:expectPhi1} in the proof of the $\lim$-type convergence result. 


\paragraph{Case 2: Let us consider now the iterations when $\| \nabla f(x_k)  \| < \zeta \delta_k $.}
First we note that if  $\| g_k \| <\eta_2\delta_k$, then we have an unsuccessful step and \eqref{eqn.unsuc.phi}  holds. Hence, we now assume that $\| g_k \| \geq \eta_2\delta_k$ and again
 consider four possible outcomes.  We will show that in all situations, except when both the model and the estimates are bad,  \eqref{eqn.unsuc.phi}   holds. In the remaining case, because $\| \nabla f(x_k)  \| < \zeta \delta_k $
 the increase in $\phi_k$ can be bounded from above by a multiple of $\delta_k^2$. Hence by selecting appropriate values for probabilities $\alpha$ and $\beta$ we will be able to establish the bound on expected decrease in $\Phi_k$ as in Case 1. 
\begin{itemize}
\item[a.] 
$I_k$ and $J_k$ are both true, i.e., both the  model and the  estimates are good on iteration $k$.  

The iteration may or may not be successful, even though $I_k$ is true. On successful iteration good model ensures reduction in $f$. Applying the same argument as in the case 1(c) we  
establish  \eqref{eqn.unsuc.phi}. 

\item[b.]  $I_k$ is true and  $J_k$ is false, i.e., we have a good model and bad estimates  on iteration  $k$. 

On unsuccessful iterations, (\ref{eqn.unsuc.phi}) holds. On successful iterations, $\| g_k \| \ge\eta_2\delta_k$ and $\eta_2\geq \kappa_{bhm}$ imply that
\begin{eqnarray*}
m_k(x_k)-m_k(x_k+s_k)&\ge& \frac{\kappa_{fcd}}{2} \| g_k \| \min \left\{ \frac{\| g_k\| }{\| H_k \|} ,\delta_k\right\}\\
&\ge&  \eta_2\frac{\kappa_{fcd}}{2}\delta_k^2.
\end{eqnarray*}
Since $I_k$ is true,  the model is $\kappa$-fully-linear, and the  function decrease can be bounded as
\begin{eqnarray*}
&&f(x_k)-f(x_k+s_k)\nonumber\\
&=&f(x_k)-m_k(x_k)+m_k(x_k)-m_k(x_k+s_k)+m_k(x_k+s_k)-f(x_k+s_k)\nonumber\\
&\ge& (\eta_2\frac{\kappa_{fcd}}{2}-2\kappa_{ef})\delta_k^2 \ge \kappa_{ef}\delta_k^2
\end{eqnarray*} 
due to \eqref{eq:cond_eta2}. 

It follows that, if $k$-th iterate is successful, then
\begin{eqnarray}
\phi_{k+1}-\phi_k\le [-\nu \kappa_{ef}+(1-\nu)(\gamma^2-1)]\delta_k^2 .\label{eqn.suc.phi.truedecrease1.case2}
\end{eqnarray}

Again by choosing $\nu\in(0,1)$  so that \eqref{eq:nu} holds,
 we ensure that 
right hand side of (\ref{eqn.suc.phi.truedecrease1.case2}) is strictly smaller than that of (\ref{eqn.unsuc.phi}), hence (\ref{eqn.unsuc.phi}) holds, whether the iteration is successful or not. 

\textbf{Remark}: $\eta_2$ may need to be a relatively large constant to satisfy \eqref{eq:cond_eta2}. This is due to the fact that the model has to be sufficiently accurate to ensure decrease in the function if a step is taken, since the step is accepted based on poor estimates. Note that $\eta_2$ restricts the size of $\Delta_k$, which is used both as a bound on the step size and the control of the accuracy. In general it is possible to have two separate quantities
(related by a constant) - one to control the step size and another to control the accuracy. Hence, it is possible to modify our algorithm  to accept steps larger than $\|g_k\|/\eta_2$. This will make the algorithm more practical, but the analysis much more complex. In this paper, we choose to stay with the simplest version, but keep in mind that the condition \eqref{eq:nu}  is not terminally restrictive.

\item[c.] $I_k$ is false and  $J_k$ is true, i.e., we have a bad model and good estimates  on iteration  $k$. 

This case is analyzed identically to the case 1(c). 

\item[d.] $I_k$ and $J_k$ are both false, i.e., both the  model and the  estimates are bad on iteration $k$.

Here we bound the maximum possible increase in $\phi_k$ using the  Taylor expansion and the  Lipschitz continuity of $\nabla f(x)$.
\begin{eqnarray*}
 f(x_k+s_k)-f(x_k)\le \|\nabla f(x_k)\|\delta_k+\frac{1}{2}L\delta_k^2<C_3\zeta\delta_k^2.
 \end{eqnarray*}
 Hence, the change in function $\phi$ is 
\begin{eqnarray}
\phi_{k+1}-\phi_k\le[\nu C_3\zeta+(1-\nu)(\gamma^2-1)]\delta_k^2.\label{eqn.suc.phi.falseincrease.case2}
\end{eqnarray}

\end{itemize}

We are now ready to bound the expectation of $\phi_{k+1}-\phi_{k}$  as we did in Case 1, except that in Case 2 we simply combine  (\ref{eqn.suc.phi.falseincrease.case2}),
which holds with probability at most $(1-\alpha)(1-\beta)$ and  (\ref{eqn.unsuc.phi}) which holds  otherwise.
\begin{eqnarray*}
&&E[\Phi_{k+1}-\Phi_{k}|\mathcal{F}_{k-1}^{M\cdot F}, \{\| \nabla f(X_k)  \| < \zeta \Delta_k\} ]\\
&\le&[\alpha\beta+\alpha(1-\beta)+(1-\alpha)\beta](1-\nu)(\frac{1}{\gamma^2}-1)\Delta_k^2 \\
&&+ (1-\alpha)(1-\beta)[\nu C_3\zeta+(1-\nu)(\gamma^2-1)]\Delta_k^2\\
&\le &(1-\nu)(\frac{1}{\gamma^2}-1)\Delta_k^2+ (1-\alpha)(1-\beta)[\nu C_3\zeta+(1-\nu)(\gamma^2-\frac{1}{\gamma^2})]\Delta_k^2
\end{eqnarray*}
if we choose probabilities $0<\alpha\leq 1$ and $0<\beta\leq 1$  so that the following holds,
\begin{eqnarray}\label{eq:alpha.beta.condition.2}
(1-\alpha)(1-\beta) \le \frac{\gamma^2-1}{\gamma^4-1 +2\gamma^2C_3\zeta \cdot \frac{\nu}{1-\nu}},
\end{eqnarray}
then 
\begin{eqnarray}\label{eq:final.decrease.in.phi.2}
E[\Phi_{k+1}-\Phi_{k}|\mathcal{F}_{k-1}^{M\cdot F}, \{\| \nabla f(X_k)  \| < \zeta \Delta_k\} ]\leq -\frac{1}{2}(1-\nu)(\frac{1}{1-\gamma^2})\Delta_k^2.
\end{eqnarray}

In conclusion, combining (\ref{eq:final.decrease.in.phi.1}) and (\ref{eq:final.decrease.in.phi.2}), and noting that $1-\frac{1}{\gamma^2}<\gamma^2-1$ we have
\begin{eqnarray*}
&E[\Phi_{k+1}-\Phi_{k}|\mathcal{F}_{k-1}^{M\cdot F}\}] \le - \frac{1}{2}(1-\nu)(1-\frac{1}{\gamma^2}) \Delta_k^2<0,
\end{eqnarray*}
which implies,  that \eqref{eq:expectationphi} holds with
$\sigma= - \frac{1}{2}(1-\nu)(1-\frac{1}{\gamma^2}-1)<0$. This 
 concludes the proof of the  theorem. 

\hfill\end{proof}

To summarize the conditions on the probabilities involved in Theorem  \ref{lemma.delta.gotozero} to ensure that the theorem holds, we state the following additional lemma.

\begin{corollary}\label{th:constants} Let all assumptions of  Theorem  \ref{lemma.delta.gotozero}  hold. The statement of Theorem \ref{lemma.delta.gotozero} holds if the $\alpha$ and $\beta$  are chosen to satisfy the following conditions: 

\begin{equation} 
\label{eq:thm411a}
\frac{(\alpha\beta-\frac{1}{2})}{(1-\alpha)(1-\beta)}\ge   \frac{1+ \frac{3L}{2\zeta}}{C_1}
\end{equation}
and
 \begin{equation}
 \label{eq:thm411b}
(1-\alpha)(1-\beta) \le
\frac{\gamma^2-1}{\gamma^4-1 + \gamma^2\left(3L+2\zeta\right)\cdot \max \left\{ \frac{4}{\zeta C_1} , \frac{4}{ \eta_1\eta_2\kappa_{fcd}}, \frac{1}{\kappa_{ef}} \right\}},
\end{equation}
with
$C_1 =\frac{\kappa_{fcd}}{4}\cdot \max\left\{ \frac{\kappa_{bhm}}{\kappa_{bhm}+\kappa_{eg}},\frac{8\kappa_{ef}}{8\kappa_{ef}+\kappa_{fcd}\kappa_{eg}}\right\}$ and $\zeta=\kappa_{eg}+\eta_2.$
\end{corollary}

\begin{proof} 
The proof follows simply from combining expression for $C_3$ and condition  \eqref{eq:nu} with (\ref{eq:alpha.beta.condition.1}) and  (\ref{eq:alpha.beta.condition.2}). 

\end{proof}

Clearly, choosing $\alpha$ and $\beta$ sufficiently close to $1$ will satisfy this condition.

\begin{remark} We will briefly illustrate through a simple example how these algorithmic parameters scale with problem data. 

Recall that $L$ is the Lipschitz constant of the gradient of $f$ and of $f$ over ${\cal L}_{enl}(x^0)$. 
It is reasonable to expect that $\kappa_{ef}$ and $\kappa_{eg}$ are quantities that scale with $L$, since Taylor models satisfy this condition, as do polynomial interpolation and regression models based on well-poised data sets \cite{DFObook}. Let us assume for the sake of an example that
$\kappa_{ef}=\kappa_{eg}=10L$. The bound on model Hessians $\kappa_{bhm}$ can be chosen to be arbitrarily small, at the expense of limiting the class of models, however, it is clearly reasonable to choose it as something that scales with $L$, if this information is available. Let us assume that $\kappa_{bhm}=10L$, as well. In a standard trust region method, a common choice of algorithmic parameters would use $\kappa_{fcd}=\frac{1}{2}$, $\gamma=2$, and $\eta_1=\frac{1}{2}$. 

The reader can easily verify that with these parameter choices and previous assumptions, Lemma \ref{th:constants} states that we must choose $\eta_2\geq 32L$. The intermediate constants satisfy $\zeta\geq 42L$ and $C_1=\frac{2}{17}$. Without loss of generality, we will simply accept $\zeta=42L$.

From observing that, given the above values of the constants,
$$
\max \left\{ \frac{4}{\zeta C_1} , \frac{4}{ \eta_1\eta_2\kappa_{fcd}}, \frac{1}{\kappa_{ef}} \right\}\leq \frac{1}{L},
$$
we have
%
\begin{equation}
\label{pracbound1}
\frac{(\alpha\beta-\frac{1}{2})}{(1-\alpha)(1-\beta)} \geq 9
\end{equation}
and
\begin{equation}
\label{pracbound2}
(1-\alpha)(1-\beta)\leq \frac{1}{440}
\end{equation}
We note that as $\eta_1$ and $\kappa_{fcd}$ constants are driven closer to $1$, the constant $440$ can be reduced by up to a factor of $4$.  

Supposing that $\kappa_{ef},\kappa_{eg}$, and $\kappa_{bhm}$ scale linearly with $L$, then $\eta_2$, $\epsilon_F$, and the expressions relating to $\alpha$ and $\beta$ in Corollary \ref{th:constants} are all  functions in $L$ satisfying
\begin{equation}
\eta_2\geq\Theta(L),
\end{equation}
\begin{equation}
\epsilon_F\leq\Theta(L),
\end{equation}
\begin{equation} 
\frac{\alpha\beta}{(1-\alpha)(1-\beta)}\geq\Theta(1),
\end{equation}
 and
 \begin{equation}
(1-\alpha)(1-\beta)\leq \Theta({1}),
\end{equation}
where we use the notation $\Theta(\cdot)$ to indicate the $O(\cdot)$ relationship with moderate constants.

\end{remark}

\begin{remark}
Recall our remark made earlier in Theorem \ref{lemma.delta.gotozero}, case 2b, on how $\eta_2$ bounds our step sizes. Indeed, if $\eta_2\geq\Theta(L)$ has to be imposed
this may force the algorithm to take small step sizes throughout. However, as mentioned earlier,  the analysis of Theorem \ref{lemma.delta.gotozero} can be modified
by introducing a tradeoff between the size of $\eta_2$ and the accuracy parameters $\epsilon_F$ and $\kappa_{ef}$ (as both of these constant parameters can be made smaller). It also may be advantageous to choose $\eta_2$ dynamically. Exploring this  is a subject for future work. In the practical implementations that we will discuss in Section 6, we do not  make use of the algorithmic parameter $\eta_2$ at all, and so even though $\eta_2$ is effectively arbitrarily close to $0$, the algorithm still works. 
\end{remark}

\begin{remark} Note that if $\beta=0$, then $\Delta_k\to 0$ for any positive  value of $\alpha$, which is the case shown in \cite{Bandeiraetal2013}, 
since, as we pointed earlier, condition \eqref{eq:alpha.beta.condition.1}
can be restated with $\frac{1}{2}$ replaced by an arbitrary small positive value.\end{remark}
\subsection{The liminf-type convergence}
We are ready to prove a liminf-type first-order convergence result, i.e., that a subsequence of the iterates drive the gradient of the objective function to zero. The proof follows closely that in \cite{Bandeiraetal2013}, the key difference being the assumption on the function estimates that are needed to ensure that a good step gets accepted by Algorithm 1. 

\begin{theorem} \label{th:convergenceliminf} 
 Let the assumptions of Theorem  \ref{lemma.delta.gotozero} and   Lemma \ref{th:constants} hold. Suppose additionally that $\alpha\beta\geq\frac{1}{2}$. Then the sequence of random iterates generated by Algorithm \ref{algo:stodfosimple}, $\{X_k\}$, almost surely satisfies $$ \underset{k\to \infty}{\lim \inf} \| \nabla f(X_k)  \|=0.$$
\end{theorem} 

\begin{proof}
We prove this result by contradiction conditioned on the almost sure event $\Delta_k \rightarrow 0$.  Let us thus assume that there exists $ \epsilon'$  such that, with positive probability, we have 
$$\|\nabla f(X_k)\|\ge \epsilon',\quad\forall k.$$

Let $ \{x_k\}$ and $\{\d_k\}$ be realizations of $\{X_k\}$ and $\{\Delta_k\}$, respectively for which $\|\nabla f(x_k)\|\ge \epsilon',\quad\forall k.$
Since $\lim\limits_{k\to \infty}\d_k=0$ (because we conditioned on $\Delta_k \rightarrow 0$), there exists $k_0$ such that for all $k\ge k_0$,
\begin{eqnarray}\label{eq:delta_small}
\d_k<b:=\min \left\{ \dfrac{\epsilon'}{2\kappa_{eg}}, \dfrac{\epsilon'}{2\kappa_{bhm}}, \dfrac{\kappa_{fcd}(1-\eta_1) \epsilon'   }{16\kappa_{ef}}, \dfrac{\epsilon'}{2\eta_2},\dfrac{\d_{\max}}{\gamma}    \right\}.
\end{eqnarray}
 We define a random variable $R_k$ with realizations $r_k=\log_{\gamma}\left( \dfrac{\d_k}{b} \right)$. Then for the realization $\{ r_k \}$ of $\{ R_k\}$,  $r_k<0$ for $k\ge k_0$. The main idea of the proof
 is to show that such realizations occur only with probability zero, hence obtaining a contradiction with the initial assumption of $\|\nabla f(x_k)\|\geq \epsilon'$ $\forall k$.  
 
 We first show that $R_k$ is a submartingale.
Recall the  events $I_k$ and $J_k$ in Definitions \ref{probabilisticmodel} and \ref{asmpesti}. 
Consider some iterate $k\ge k_0$ for which $I_k$ and $J_k$ both occur, which happens with probability $P(I_k\cap J_k)\ge  \alpha\beta$.  Since \eqref{eq:delta_small} holds we have exactly the same 
situation as in Case 1(a) in the proof of Theorem  \ref{lemma.delta.gotozero}. In other words, we can apply Lemmas   \ref{lemma.delta.2} and \ref{lemma.delta.3} to conclude that the $k$-th iteration is successful,
hence, the  trust-region radius is increased. In particular, since $\d_k\le \frac{\d_{\max}}{\gamma}$,  $\d_{k+1}=\gamma\d_{k}$. Consequently, $r_{k+1}=r_k+1$.
 \vspace{0.1in}

Let $\mathcal{F}_{k-1}^{I\cdot J}=\sigma({I_0},\cdots, {I_{k-1}})\cap \sigma  ({J_0},\cdots, J_{k-1})$. For all other outcomes of $I_k$ and $J_k$, which occur with total probability of at most  $1-\alpha\beta$, we have $\d_{k+1}\ge \gamma^{-1}\d_{k}$.  Hence
\begin{eqnarray*}
&\Embb[ r_{k+1}| \mathcal{F}_{k-1}^{I\cdot J}] \geq\alpha\beta(r_k+1)+(1-\alpha\beta)(r_k-1)\ge r_k,
\end{eqnarray*}
as long as $\alpha\beta\geq 1/2$, which implies that $R_k$ is a submartingale. 

Now let us construct another submartingale $W_k$, on the same probability space as $R_k$ which will serve as a lower bound on $R_k$ and for which  $\left\{\underset{k\to \infty}{\lim\sup}\ W_k=\infty\right\}$ holds almost surely.
Define indicator random variables $\textbf{1}_{I_k}$ and $\textbf{1}_{J_k}$ such that $\textbf{1}_{I_k}=1$ if $I_k$ occurs, $\textbf{1}_{I_k}=0$ otherwise, and similarly, $\textbf{1}_{J_k}=1$ if $J_k$ occurs, $\textbf{1}_{J_k}=0$ otherwise. Then define
\begin{eqnarray*}
W_k=\sum\limits_{i=0}^k(2\cdot \textbf{1}_{I_k}\cdot \textbf{1}_{J_k}-1).
\end{eqnarray*}
Notice that $W_k$ is a submartingale since 
\begin{eqnarray*}
\Embb[ W_k|\mathcal{F}_{k-1}^{I\cdot J}  ] &=&E[ W_{k-1}|\mathcal{F}_{k-1}^{I\cdot J}  ] +E[2\cdot \textbf{1}_{I_k}\cdot \textbf{1}_{J_k}-1| \mathcal{F}_{k-1}^{I\cdot J}  ]\\
&=& W_{k-1} +  2E[ \textbf{1}_{I_k}\cdot \textbf{1}_{J_k}| \mathcal{F}_{k-1}^{I\cdot J}  ]  -1\\
&=& W_{k-1}  +2P( I_k\cap J_k |\mathcal{F}_{k-1}^{I\cdot J} )-1\\
&\ge& W_{k-1},
\end{eqnarray*}
where 
the last inequality holds because $\alpha\beta\geq 1/2$. 
Since $W_k$ only has $\pm 1$ increments, it has no finite limit. Therefore, by Theorem \ref{submartingale}, we have $\left\{\underset{k\to \infty}{\lim\sup}\ W_k=\infty\right\}$.

By the construction of $R_k$ and $W_k$, we know that $r_k-r_{k_0}\ge w_k -w_{k_0}$. Therefore, $R_k$ has to be positive infinitely often with probability one. This implies that the sequence of realizations $r_k$ such that $r_k<0$
 for $k\ge k_0$ occurs with probability zero. Therefore our assumption that $ \| \nabla f(X_k)  \|\geq \epsilon '$  hold for all $k$ with positive probability is false and $$ \underset{k\to \infty}{\lim \inf} \| \nabla f(X_k)  \|=0 $$ holds
almost surely.

\hfill\end{proof}

\subsection{The lim-type convergence}

In this subsection we show that $\lim_{k\to\infty} \|\nabla f(X_k) \|=0$ almost surely.
%

We now state an auxiliary lemma, which is similar to the one in  \cite{Bandeiraetal2013}, but requires a different proof because in our case the function values $f(X_k)$ can increase with $k$, while
in the case considered in \cite{Bandeiraetal2013}, function values are monotonically nonincreasing. 

\begin{lemma}\label{auxLemma_to_theorem:DFO_Random_SIMPLE1_convLIM}
 Let the same assumptions that were made in Theorem \ref{th:convergenceliminf} hold. 
Let $\{ X_k \}$ and $\{ \Delta_k \}$ be sequences of random iterates and
random trust-region radii generated
by Algorithm~\ref{algo:stodfosimple}. Fix $\epsilon>0$ and define the sequence $\left\{K_\epsilon\right\}$ consisting of the natural numbers $k$ for which $\|\nabla f(X_k)\| >\epsilon$ (note that $K_\epsilon$ is a sequence of random variables). Then,
$$\sum_{k\in \{K_\epsilon\}}{\Delta_k} \; < \; \infty$$
almost surely.
\end{lemma}

\begin{proof}

From Theorem  \ref{lemma.delta.gotozero} we know that $\sum \Delta_k^2 <\infty$ and hence $\Delta_k\to 0$ almost surely. 
For each  realization of Algorithm~\ref{algo:stodfosimple} and a sequence $\{\delta_k\}$, 
there exists $k_0$ such that  $\delta_k\leq \epsilon/\zeta$, $\forall k\geq k_0$, where $\zeta$ is defined as in Theorem  \ref{lemma.delta.gotozero}. Let $K_0$ be the random variable with realization $k_0$ and let 
$K$ denote the sequence of indices $k$ such that $k\in K_\epsilon$ and $k\geq K_0$. 
Then for all $k\in K$, Case 1 of 
Theorem  \ref{lemma.delta.gotozero} holds, i.e., $\| \nabla f(X_k)  \| \geq \zeta \Delta_k$, since $
\| \nabla f(X_k)  \|\geq \epsilon$ for all $k\in K$. From this and from \eqref{eq:expectPhi1} we have
\[
E[\Phi_{k+1}-\Phi_{k}|\mathcal{F}_{k-1}^{M\cdot F}]  \le - \frac{1}{2}C_1\nu \epsilon \Delta_k, \quad\forall k\geq k_0.
\]

Recall that $\Phi_{k}$ is bounded from below. Hence, summing up the above inequality for all  $k\in K$ and taking the expectation, we have that
$$\sum_{k\in K}{\Delta_k} \; < \; \infty$$
almost surely. Since $K_\epsilon\subseteq K\cup \{k\leq K_0\}$ and $K_0$ is finite almost surely then the   statement of the lemma holds.
\end{proof}

%

We are now ready to state the $\lim$-type result.

\begin{theorem}\label{theorem:DFO_Random_simple_convLIM}
Let the same assumptions as in Theorem \ref{th:convergenceliminf} hold. 
Let $\{ X_k \}$ be a sequence of random iterates generated
by Algorithm~\ref{algo:stodfosimple}. Then, almost surely,
\[
\lim_{k\to\infty} \|\nabla f(X_k)\| \; = \; 0.
\]
\end{theorem}

{\proof The proof of this result, is almost  identical to the proof of the same theorem in  \cite{Bandeiraetal2013} hence we will not present the proof here. The key idea of the proof is to show that if the theorem does not hold, then with positive probability $$\sum_{k\in \{K_\epsilon\}}{\Delta_k} \; = \; \infty,$$
with $K_\epsilon$ defined as in Lemma \ref{auxLemma_to_theorem:DFO_Random_SIMPLE1_convLIM}. This result is shown using Lipschitz continuity of the gradient and does not depend on the stochastic nature of the algorithm. Since this result contradicts the almost sure result of Lemma \ref{auxLemma_to_theorem:DFO_Random_SIMPLE1_convLIM}, we can conclude that the statement of the theorem holds almost surely. }

\section{Constructing models and estimates in different stochastic settings.}\label{sec:learningbounds}
We now discuss various settings of stochastic noise in the objective function and how $\alpha$-probabilistically
$\kappa$-fully linear models and $\beta$-probabilistically $\epsilon_F$-accurate estimates can be obtained in these settings.

Recall that we assume that for each $x$ we can compute the value of $\tilde{f}(x)$, which is the noisy  version of $f$, 
$$\tilde{f}(x) = f(x,\omega),$$
where  $\omega$ is a random variable which induces the noise.

\paragraph{Simple stochastic noise.} The example of noise which is typically considered in stochastic optimization
is the ``i.i.d." noise,  that is noise with distribution independent of the function values. Here we consider a somewhat more general setting, where 
the noise is unbiased for all $f$, i.e., 
$$E_\omega[{f}(x,\omega)] = f(x), \quad \forall x, $$
and
$$\text{Var}_\omega[{f}(x,\omega)] \leq V<\infty, \quad \forall x. $$

This is the typical noise assumption in stochastic optimization literature. 
In the case of unbiased noise as above, constructing estimates and models that satisfy our assumptions is fairly straight-forward. 
First, let us consider the case when only the noisy function values are available (without any gradient information), where we want to construct a model that is $\kappa$-fully linear in a given trust region $B(x^0, \delta)$ with some reasonably large probability, $\alpha$. 

One can employ standard sample averaging approximation techniques to reduce the variance of the function evaluations. In particular, let $\bar {f}_p(x,\omega)= \frac{1}{p} \sum_{i=1}^p{f}(x,\omega_i)$, where 
$\omega_i$ are the
i.i.d. realizations of the noise $\omega$.  Then, by Chebyshev inequality, for any $v>0$, 
\[
P(|\bar {f}_p(x,\omega)-f(x)]| >  v )=P(|\bar {f}_p(x,\omega)-E_\omega[{f}(x,\omega)]| >  v )\leq \frac{V}{pv^2}.
\]
In particular,  we want $v= \kappa^\prime_{ef}\delta^2$ for some $\kappa^\prime_{ef}>0$ and $\frac{V}{pv^2}\leq 1-\alpha^\prime$ for some $\alpha^\prime$, which can be ensured by choosing  
$p\geq \frac{V}{(\kappa^\prime_{ef})^2(1-\alpha^\prime) \delta^4}$. 

We  now construct a fully linear model as follows: given a  well-poised set\footnote{See \cite{DFObook} for details on well-poised sets and how they can be obtained.} $Y$ of $n+1$ points in $B(x^0, \delta)$, at each point $y^i\in Y$, we  compute
$\bar{f}_p(y^i, \omega )$ and build a  linear interpolation model $m(x)$ such that $m(y^i)=\bar{f}_p(y^i, \omega )$, for all 
$i=1, \ldots, n+1$. Hence, for any $y^i\in Y$, we have
\[
P(|m(y^i)-{f}(y^i)] |>  \kappa^\prime_{ef} \delta^2) \leq 1-\alpha^\prime.
\]
Moreover, the events $\{|m(y^i)-f(y^i)]| >  \kappa^\prime_{ef} \delta^2\}$ are independent,
hence 
\[
P(\max_{i=1..n+1} \{|m(y^i)-{f}(y^i)]|\} >  \kappa^\prime_{ef} \delta^2) \leq 1-(\alpha^\prime)^{n+1}. 
\]
It is easy to show using, for example, techniques described in \cite{DFObook}, that $m(x)$ is a $\kappa$-fully linear model of $E_\omega[{f}(x,\omega)] $
in $B(x^0, \delta)$ for appropriately chosen 
$\kappa=(\kappa_{eg}, \kappa_{ef})$, with probability at least $\alpha=(\alpha^\prime)^{n+1}$.

Computing the $\beta$-probabilistically $\epsilon_F$-accurate estimates of $f(x, \omega)$ can be done analogously to the construction of the models described above.


The majority of stochastic optimization and sample average approximation methods focus on
derivative based optimization where it is assumed  that, in addition to  $f(x, \omega)$, $\nabla_x f(x, \omega)$ is also available, 
 and that the noise in the gradient computation is also
independent of $x$, that is
$$E_\omega[\nabla_x{f}(x,\omega)] =\nabla f(x), \quad \forall x  $$
and 
$$\text{Var}_\omega[\|\nabla_x{f}(x,\omega)\|] \leq V<\infty, \quad \forall x, $$
(in general the variance of the gradient and the function value  are not the same, but here for simplicity we bound both by $V$). 

In the case when the noisy gradient values are available, the construction of fully linear models in $B(x^0, \delta)$ is  simpler. 
Let $\bar\nabla {f}_p(x,\omega)= \frac{1}{p} \sum_{i=1}^p{\nabla f}(x,\omega_i)$. 
Again, by extension of Chebychev inequality, for $p$ such that
\[
p\geq \max\{\frac{V}{\kappa^2_{ef}(1-\alpha^\prime) \delta^4}, \frac{V}{\kappa^2_{eg}(1-\alpha^\prime) \delta^2}\},
\]
\[
P(\|\bar \nabla f_p(x^0,\omega)-\nabla f(x^0)] \|>  \kappa_{eg}\delta )=P(\|\bar \nabla f_p(x^0,\omega)-E_\omega[{\nabla f}(x^0,\omega)] \|>  \kappa_{eg}\delta )\leq 1-\alpha^\prime, 
\]
and 
\[
P(|\bar {f}_p(x^0,\omega)-{f}(x^0)| >  \kappa_{ef}\delta^2)=P(|\bar {f}_p(x^0,\omega)-E_\omega[{f}(x^0,\omega)]| >  \kappa_{ef}\delta^2)\leq 1- \alpha^\prime.
\] 
Hence  the linear expansion
$m(x)=\bar {f}_p(x^0,\omega)+\bar \nabla f_p(x^0, \omega)^T(x-x^0)$ is a $\kappa$-fully linear model of $f(x)=E_\omega[{f}(x,\omega)] $
on $B(x^0, \delta)$ for appropriately chosen $\kappa=(\kappa_{eg}, \kappa_{ef})$, with probability at least $\alpha=(\alpha^\prime)^2$.

In \cite{larson2013stochastic} it is shown that  least squares regression models based on sufficiently large strongly poised \cite{DFObook} sample sets 
are $\alpha$-probabilistically $\kappa$-fully linear models.

There are many existing methods and convergence results using sample average approximations and stochastic gradients for stochastic optimization with i.i.d. or unbiased noise. 
Some of these methods have been shown to achieve optimal sampling rate \cite{Pasupathyetal}, that is they converge to the optimal solution while sampling the gradient at the best possible rate. We do not provide convergence rates in this paper (it is a subject for future research), hence it remains to be seen if our algorithm can achieve the  optimal rate. Our contribution here is the method which applies beyond the i.i.d. case, as we will discuss below. In Section \ref{sec:computations}, however, we demonstrate that our method can have superior numerical behavior compared to standard sample averaging even in the case of the i.i.d. noise, so it is at least competitive in practice. 

\paragraph{Function computation failures.}
Now, let us consider a more complex noise case. Suppose that the function $f(x)$ is computed (as it often is, in black-box optimization) by some numerical method
that includes a random component.  Such examples are common in machine learning applications, for instance, where $x$ is the vector of hyperparameters of a learning algorithm and $f(x)$ is the expected error of the resulting classifier. In this case, 
for each value of $x$, the classifier is obtained by solving an optimization problem, up to a certain accuracy, on a given training set. Then the classifier is evaluated on the testing set. If a randomized coordinate descent or a stochastic gradient descent is applied  to train the classifier for a given vector $x$, then the resulting classifier
is sufficiently close to the optimal classifier with some known probability. However, in the case when the training of the classifier fails to produce a sufficiently  accurate 
solution, the resulting error is difficult to estimate.  Usually, it is possible to know the upper bound on the value of this inaccurate  objective function  but nothing else may be known about the distribution  of this value. Moreover, it is likely that the probability of the accurate computation depends on  $x$; for example, after $k$ 
iterations  of randomized coordinate descent, the error between the true $f(x)$ and the computed $\tilde f(x)$ is bounded by some value, with probability $\alpha$, where the value depends on $\alpha$, $k$, and $x$ \cite{RichtarikTakac}. 

Another example is solving a system of nonlinear black-box equations. Assume that we seek $x$ such that $\sum_i (f_i(x))^2=0$, for some functions $f_i(x)$, $i=1, \ldots, m$ that are computed by numerical simulation, with noise. 
As is often done in practice (and is supported by our theory) the noise in the function computation is reduced as the algorithm progresses, for example, by reducing the size of a discretization, step size, or convergence tolerance within the black-box computation. These adjustments for noise reduction
 usually 
 increase the workload of the simulation.  With the increase of the workload, there is an increased probability of failure of the code. Hence, the smaller the values of $f_i(x)$, the more likely the computation of $f_i(x)$ will fail and some inaccurate value is returned. 
 
 These two examples show that the noise in $\tilde f(x)$ may be large with some positive probability, which may depend on $x$. 
 Hence, let us  consider the following, idealized, noise model
$$
\tilde{f}(x) = f(x,\omega)=\left \{ \begin{array}{ll} f(x), & {\rm w.p. } \ 1-\sigma(x)\\ \omega(x)\leq V & {\rm w.p.} \ \sigma(x), \end{array}\right.
$$
where $\sigma(x)$ is the probability with which the function $f(x)$ is computed inaccurately, and $\omega(x)$ is some random function of $x$, for which 
only an upper bound $V$ is known. This case is idealized, because we assume that with probability $1-\sigma(x)$, $f(x)$ is computed exactly. It is trivial to extend
this example to the case when $f(x)$ is computed with an error, but this error can be made sufficiently small.

For this model of function computation failures we have
\[
E_\omega[{f}(x,\omega)] =(1-\sigma(x))f(x)+\sigma(x)E [\omega(x)]\neq  f(x), \quad \forall \sigma(x)>0.
\]
and it is clear, that  for any $\sigma(x)>0$, unless $E[\omega(x)]\equiv \text{some constant}$, optimizing $E_\omega[{f}(x,\omega)]$ does not give the same result as optimizing $f(x)$. 
Hence applying Monte-Carlo sampling within an optimization algorithm solving this problem is not a correct approach. 

We now observe that constructing $\alpha$-probabilistically
$\kappa$-fully linear models and $\beta$-probabilistically $\epsilon_F$-accurate estimates is trivial in this case, assuming that $\sigma(x)\leq \sigma$ for all $x$, when $\sigma$ is small enough. 
In particular, given a trust region $B(x^0, \delta)$, sampling a function $f(x)$ on a sample set $Y\subset B(x^0, \delta)$ well-poised for linear interpolation will produce a $\kappa$-fully linear model in $B(x^0, \delta)$ with probability at least $(1-\sigma)^{|Y|}$, since with this probability all of the function values are computed exactly. Similarly,
for any $s\in B(x^0, \delta)$, the function estimates $F^0$ and $F^s$ are both correct with probability at least $(1-\sigma)^2$. 
Assuming that $(1-\sigma)^{|Y|}\geq \alpha$ and $(1-\sigma)^2\geq \beta$, where $\alpha$ and $\beta$ satisfy the assumptions of Theorem  \ref{lemma.delta.gotozero}  and Lemma \ref{th:constants} and $\alpha\beta\geq\frac{1}{2}$ as in Theorem \ref{th:convergenceliminf},
we observe that the resulting models satisfy our theory. 

\begin{remark} We assume here that the probability of failure to compute $f(x)$ is small enough for all $x$. 
In the machine learning example above, it is often possible to control the probability $\sigma(x)$ in the computation of $f(x)$, for example by increasing the number of 
iterations of a randomized coordinate descent or stochastic gradient method.  In the case of the black-box nonlinear equation solver, the probability of code failure
is expected to be quite small. There are, however, examples  of black box optimization problems where the computation 
of $f(x)$ fails all the time for specific values of $x$. This is often referred to as hidden constraints \cite{MoreW09}. 
 Clearly our theory does not apply here, but we believe there is no local method that can provably converge to a local minimizer in such a setting without additional information about these specific
 values of $x$. 
\end{remark}

\section{Computational Experiments} \label{sec:computations}

In this section, we will discuss the performance of several variants of our proposed method (varied in the way the models are constructed), 
henceforth only referred to as STORM (STochastic Optimization using Random Models), 
that target various noisy situations discussed in the previous section. 
We note that a comparison  of  STORM 
to the SPSA method of \cite{SPSA} and the classical Kiefer-Wolfowitz method in \cite{KieferWolfowitz} has been reported in 
 \cite{RChen_2015} and shows that
STORM significantly outperformed these two methods, while no special tuning of  SPSA or Kiefer-Wolfowitz was
applied. Since 
a trust-region based method, which is  able to use second order information is likely to 
outperform stochastic gradient-like methods in many settings, we omit such comparison here. 

Throughout this section, all proposed algorithms were implemented in Matlab and all experiments were performed on
a laptop computer running Ubuntu 14.04 LTS with an Intel Celeron 2955U @ 1.40GHz dual processor.  

\subsection{Simple stochastic noise}
In these experiments, we used a set of 53 unconstrained problems adapted from the CUTEr test set, 
each being in the form of a sum of squares problem, i.e.

\begin{equation}
\label{ssq}
f(x) = \displaystyle\sum_{i=1}^m (f_i(x))^2,
\end{equation}
where for each $i\in \{1,\dots,m\}$, $f_i(x)$ is a smooth function. 
Two different types of noise will be used in this first subsection, which we will refer to as \emph{multiplicative} noise 
and \emph{additive} noise. In the multiplicative noise case, for each $i\in\{1,\dots,m\}$, we generate some $\omega_i$ 
from the uniform distribution on $[-\sigma,\sigma]$ for some parameter $\sigma>0$, and then compute the noisy function

\begin{equation}
\label{ssq:rel}
\tilde f(x,\omega) = \displaystyle\sum_{i=1}^m( (1+\omega_i)f_i(x))^2.
\end{equation}

The key characteristic of this noise is that for each $x$, we have $E_{\omega}[f(x,\omega)] = f(x)$, 
however the variance is nonconstant over $x$ and scales with the magnitudes of the components $f_i(x)$. 
Thus, one should expect that if an algorithm is minimizing a function of the form \eqref{ssq:rel},
the accuracy of the
estimates of $f(x)$ based on a constant  number of samples of $\tilde f(x,\omega)$ should increase,
assuming that the algorithm produces a decreasing sequence $\{f(x^k)\}_{k=1}^\infty$. 
While this behavior is not supported by theory, because we do not know how quickly $f(x^k)$ decreases, 
our computational results show that a constant number of samples is indeed sufficient for convergence. \\

The other type of noise we will test is \emph{additive}, i.e. we additively perturb each component in \eqref{ssq} 
by some $\omega_i$ uniformly generated in $[-\sigma,\sigma]$ for some parameter $\sigma>0$. That is,

\begin{equation}
\label{ssq:add}
\tilde f(x,\omega) = \displaystyle\sum_{i=1}^m (f_i(x)+\omega_i)^2
\end{equation}

Note that the noise is additive only in terms of the component functions, but not in terms of the objective function, moreover
$E_{\omega}[f(x,\omega)] = f(x)+\sum_{i}^mE(\omega_i)^2$. However, the constant bias  term does not affect optimization results, since
$\min_x E_{\omega}[f(x,\omega)] = \min_x f(x)$.

In our first set of experiments  for these two noisy settings, we compare  a version of STORM to a version of sample average 
based trust region algorithms, which we will call ``TR-SAA'', and which is similar to a trust-region algorithm presented in \cite{DengFerris}. 
Similar method, with convergence guarantees, has been recently proposed in \cite{2015sarhasetal}.
 In their work, they  use a Bayesian scheme 
to select a sufficiently large sample complexity for computing average function values at a current interpolation set. 
Here, in TR-SAA, we simplify this approach, by  increasing
sample complexity in each iteration proportionally to the decrease of the trust region radius $\delta_k$. 
A description of TR-SAA is given in Algorithm \ref{tr-saa} in the Appendix. 

There are two particular aspects of TR-SAA that we would like to draw attention to: in the estimate calculation step,
the computation of $f_k^0$ is performed {\em before}  the model $m_k$ is constructed, and $m_k$ is built to interpolate $f_k^0$, 
hence the quality of estimate $f_k^0$ and that of the model  $m_k$ are dependent.  Additionally,
the quality of the model $m_k$ is dependent on that of $m_{k-1}$ because the samples are reused.
Both of these aspects are violations of the typical assumptions of STORM. 
Thus, we also propose TR-SAA-resample, which is the same algorithm as TR-SAA except that at each iteration, the function value
at every interpolation point is recomputed as an average of function evaluations, independent of past function evaluations.
While TR-SAA-resample may overcome some of the problems of the dependence of $m_k$ on $m_{k-1}$, it still doesn't satisfy
the assumptions of STORM because of the dependence of $f_k^0$ on $m_k$. 

Thus, in Algorithm \ref{STORM-addrel}, stated in the Appendix,
we propose a version of STORM, comparable to TR-SAA in terms of sample sizes.
In Algorithm \ref{STORM-addrel}, the models $m_k$ and $m_{k-1}$ are entirely independent since a new regression set is drawn
in each iteration. Additionally, in the estimates calculation step, the computations of $f_k^0$ and $f_k^s$ are completely
independent of the model $m_k$. For these reasons, Algorithm \ref{STORM-addrel} is more in line with the theory analyzed
in this paper than a sample average approximation scheme like in Algorithm \ref{tr-saa}. 

For each of the 53 problems, the best known value of the noiseless $f(x)$ obtained by a solver is recorded as $f^*$. 
We recorded the number of function evaluations required by a solver 
to obtain a function value $f(x^k) < f'$ such that
\begin{equation}
\label{tau_sufficiently_solved}
1-\tau < \displaystyle\frac{f(x^0)-f'}{f(x^0)-f*}.
\end{equation}
This number was averaged over 10 runs for each problem.
In the profiles shown in Figure \eqref{fig:multaddnoise} for the multiplicative noise case $\tau=10^{-3}$.
 In all the experiments, a budget of $1000(n+1)$ noisy function evaluations was set.
For the choice of initialization, the same parameters were used in all of TR-SAA, TR-SAA-resample, and STORM-unbiased: 
$\delta_{\max} = 10, \delta_0=1, \gamma=2,\eta_1=0.1,\eta_2=0.001,p_{\min} = 10.$

\begin{figure}\label{fig:multaddnoise}
\begin{center}
\includegraphics[height=4.9cm]{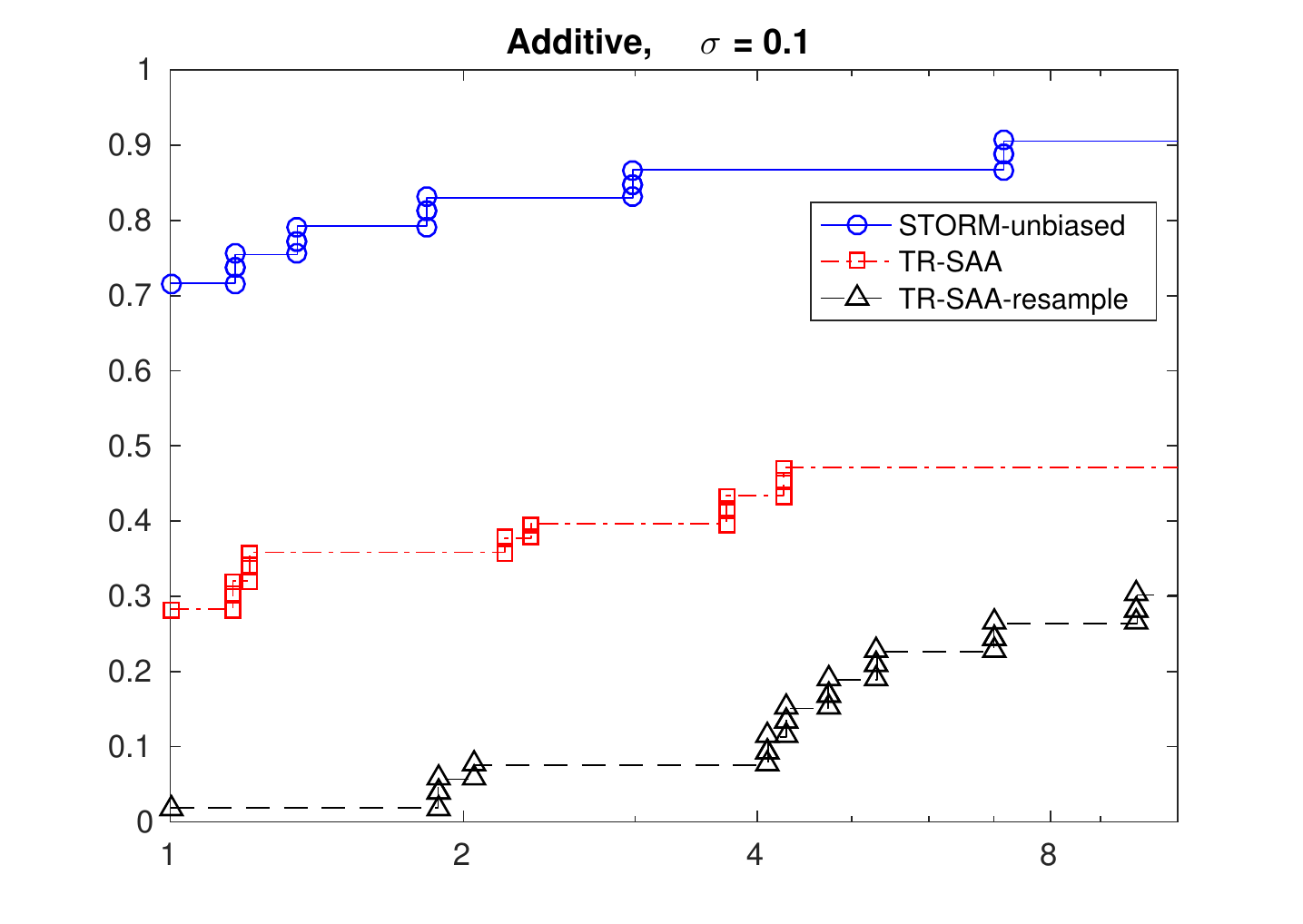}
\includegraphics[height=5cm]{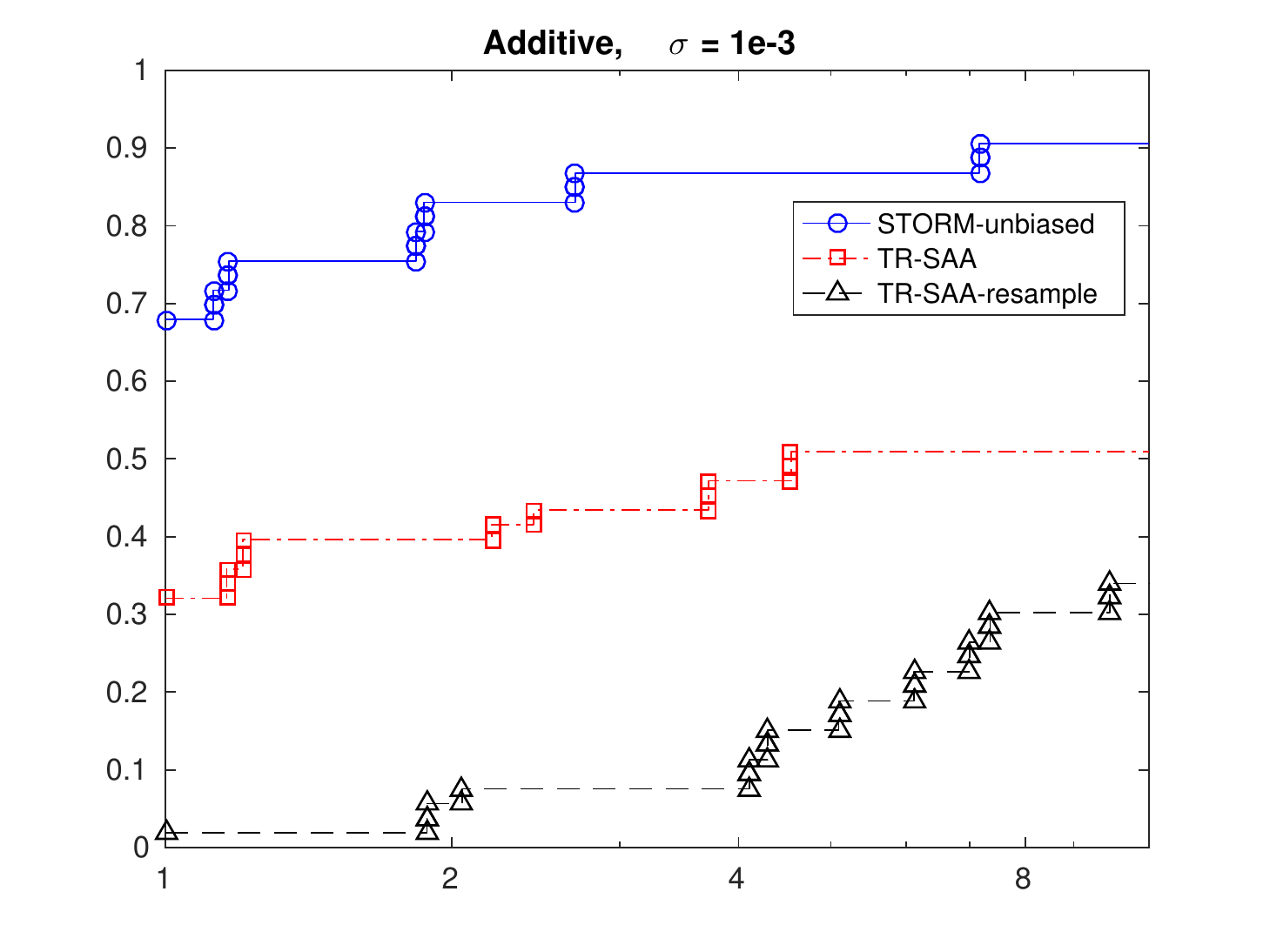}\\

\includegraphics[height=5cm]{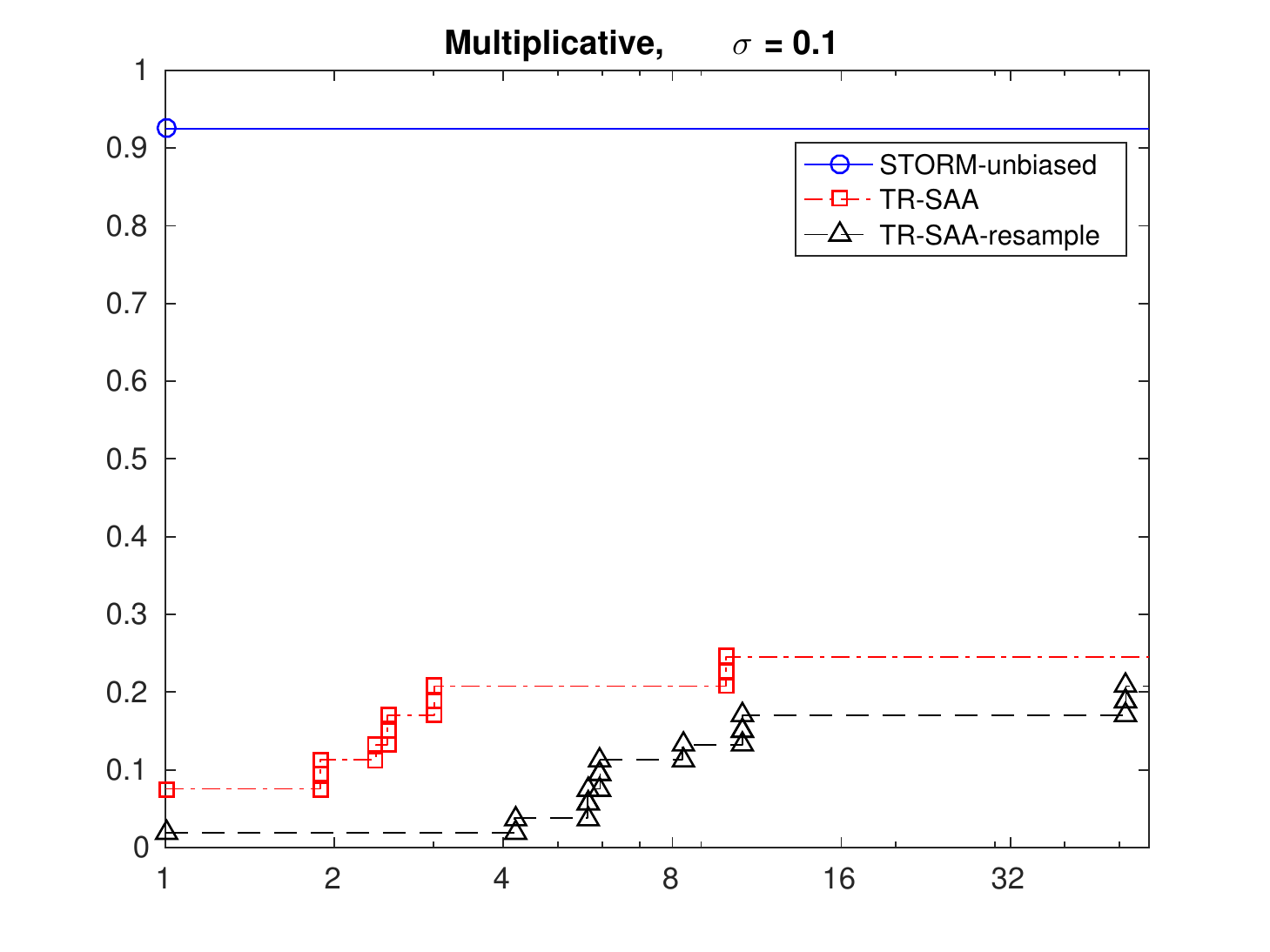}
\includegraphics[height=5cm]{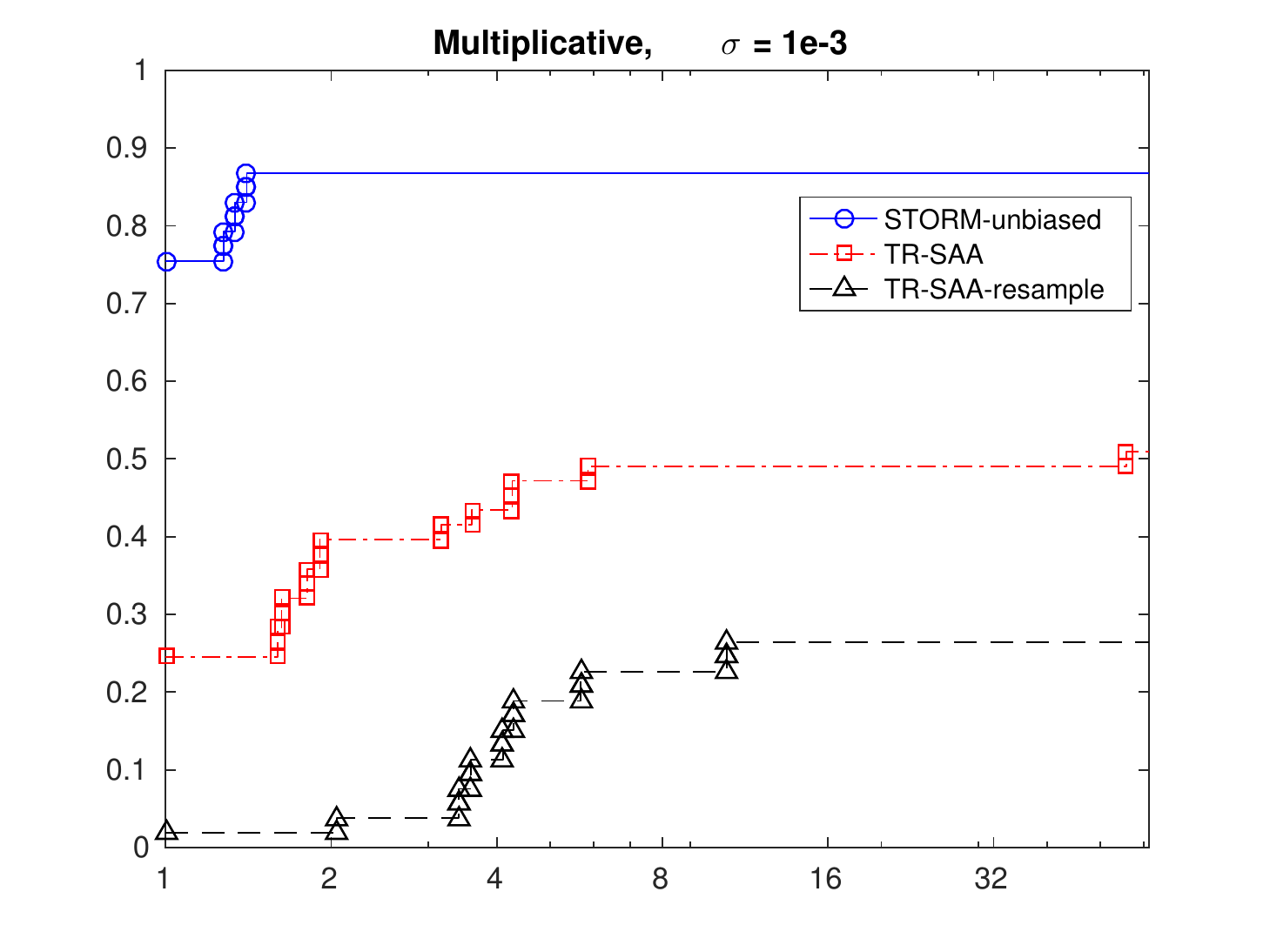}
\end{center}
\caption{Performance profiles ($\tau=10^{-3}$) comparing STORM-unbiased, TR-SAA, and TR-SAA-resample on the testset.}
\end{figure}

Note that even though we have ignored the theoretical prescription derived in the previous section that sample rate should scale with $1/\delta_k^4$,
we note that STORM-unbiased performs extremely well compared to the TR-SAA method. Although we chose to sample at a rate so that
$p_k$ was on the order of $1/\delta_k$, this particular sample rate was chosen after testing various other rates on the same set of test functions,
and seemed to work relatively well for both STORM-unbiased and TR-SAA. 



\paragraph{Function computation failures.} In these experiments, we used the same 53 sum of squares problems as in the unbiased noise experiments described above, 
but introduced  biased noise. For each component in the sum in (\ref{ssq}), if $|f_i(x)| < \epsilon$ for some parameter $\epsilon>0$, then $f_i(x)$ is computed as

\begin{equation*}
f_i(x) = 
\left\{
\begin{array}{ll}
f_i(x) & \text{w.p.}\hspace{.5pc}1-\sigma\\
V & \text{w.p.} \hspace{.5pc}\sigma\\
\end{array}
\right.
\end{equation*}
for some parameter $\sigma>0$ and for some ``garbage value" $V$. If $f_i(x)\geq\epsilon$, then it is deterministically computed as $f_i(x)$. This noise is biased, with bias depending on $x$, and we should not expect any sort of averaging approximation to work well here. This is indeed indicated in our experiments, where various levels of $\sigma$ and $\epsilon$ are shown below.  
 There  choice of $V$ did not significantly affect the results, and in the experiments illustrated below $V=-10000$ was used. 
 Obviously, the intention here is that such a large negative value will cause STORM to see a trial step as promising, 
 when it may, in fact, yield an increase in function value if taken. 
 We propose the version of STORM presented as Algorithm \ref{STORM-failure} in the appendix.  

The key feature of Algorithm \ref{STORM-failure} is that on each iteration, the interpolation set changes minimally as in
a typical DFO trust region method, but the interpolated function values are computed afresh. Intuitively, this is the right
thing to do, since if a ``garbage value'' is computed at some point in the algorithm to either construct a model or provide
a function value estimate, 
we do not want its presence to affect the computation of models in subsequent iterations. No averaging is performed, as it can only cause harm in the setting.

Since we are not aware of any other optimization algorithm that is designed for this case of noise, we performed no comparisons, but 
experiment to discover how the method works as a function of the probability of failure.
On the test set of 53 problems, we ran Algorithm \ref{STORM-failure} 30 times and report the average percentage of instances that are
solved in the sense of \eqref{tau_sufficiently_solved} with $\tau=10^{-3}$ within a budget of $10000(n+1)$ function evaluations,
where $f^*$ was computed by Algorithm \ref{STORM-failure} with $\sigma=0$. In order to standardize the probability of failure over the test set,
we define the probability of success $p_s$ and then take $\sigma$ on a function with $m$ component to be 
$\sigma = 1-p_s^{(1/m)}$. The results are summarized in Figure \ref{fig:comp_fail}. 

\begin{figure}
\centering
\label{fig:comp_fail}
\caption{Average percentage of problems solved as a function of probability of success $p_s$.}
 \includegraphics[height=5cm]{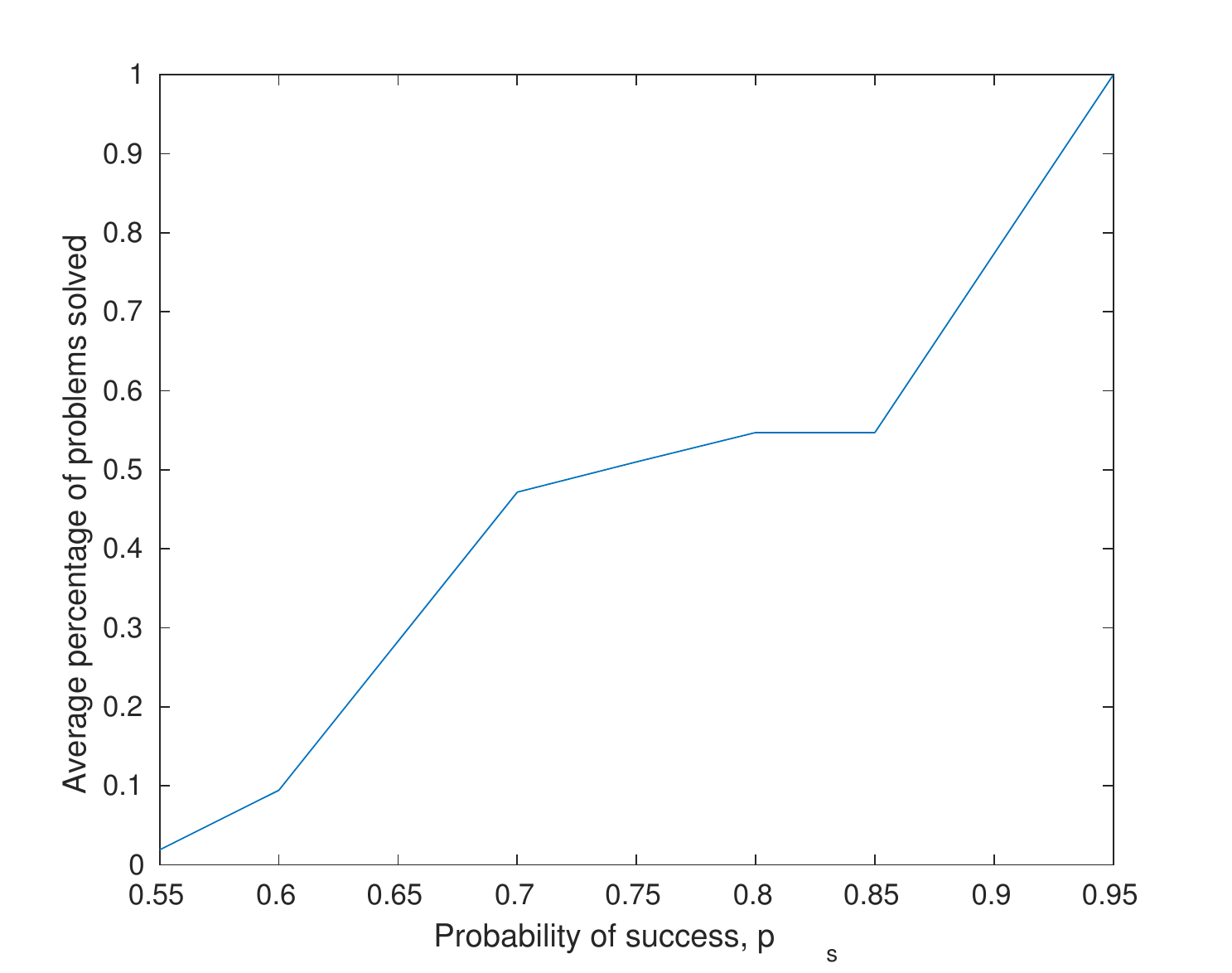}
\end{figure}

 
 This experiment suggests that the practical threshold at which STORM fails to make progress may be looser than that suggested
 by theory.  We will illustrate this idea through a simple example. 
Consider the minimization of the simple quadratic function

\begin{equation}
\label{simple_quadratic}
f(x) = \displaystyle\sum_{i=1}^n (x_i-1)^2.
\end{equation}

The minimizer uniquely occurs at the vector of all 1s. 
Now consider the minimization of this function under our setting of computation failure 
where we vary the probability parameter $\sigma$ and fix $\epsilon = 0.1$. 
Suppose on each iteration of STORM,  the interpolation set contains $(n+1)(n+2)/2$ points. 
Then, the probability of obtaining the correct quadratic model in the worst case where for all $i$, 
$|x_i-1|<\epsilon$  is precisely $\alpha = ((1-\sigma)^n)^{\frac{(n+1)(n+2)}{2}}$. 
Likewise, the probability of obtaining the correct function evaluation for $F_0$ and $F_s$ on each iteration 
in the worst case is $\beta = ((1-\sigma)^n)^2$. Now, supposing we  initialize STORM with the zero vector 
in $\mathbb{R}^n$, it is reasonable to assume that all iterates will occur near the unit cube $[0,1]^n$, 
and so we can use simple calculus to estimate a Lipschitz constant of the function over the relevant domain 
as $2\sqrt{n}$, and the Lipschitz constant of the gradient is constantly $2$.
These also serve as reasonable estimates of $\kappa_{ef}$ and $\kappa_{eg}$, respectively. 
Using parameter choices $\gamma=2$, $\eta_1=0.1$, $\eta_2=1$, 
we can use the bounds in \eqref{eq:thm411a} and \eqref{eq:thm411b} to
solve for the smallest allowable $(1-\sigma)$ for which our algorithm can guarantee convergence. 
For $n=2$, our theory suggests that we can safely lower bound $(1-\sigma) > 0.9592$ 
(which implies $\alpha\approx 0.6069$ and $\beta\approx 0.8467$), 
and for $n=10$, we can lower bound $(1-\sigma) > 0.9990$ 
(which implies $\alpha\approx 0.5233$ and $\beta\approx 0.9806$).

In Figure \ref{fig:simple_quad} for $n=2,10$, we plot an indicated level of $(1-\sigma)$ on the $x$-axis against the proportion of 100 randomly seeded instances with that level of $(1-\sigma)$ 
that Algorithm \ref{STORM-failure} managed to find a solution $x^*$ satisfying $f(x^*)< 10^{-5}$ within $10^4$ many
function evaluations using the discussed parameter choices. 
The red line shows the level of $(1-\sigma)$ that our theory predicted in the previous paragraph. 
As we can see, $(1-\sigma)$ can be quite smaller than predicted by our theory before the 
failure rate becomes unsatisfactory. 
As a particular example, in the $n=10$ case, when $(1-\sigma) = .998$, the corresponding probabilities are $\alpha\approx 0.266782$ and $\beta\approx  0.960751$, and yet $100\%$ of the instances were solved to the required level of accuracy. In other words, even though the models are eventually only accurate on roughly $27\%$ of the iterations, we still see satisfactory performance.\\

\begin{figure}
\label{fig:simple_quad}
\begin{center}
\includegraphics[height=6cm]{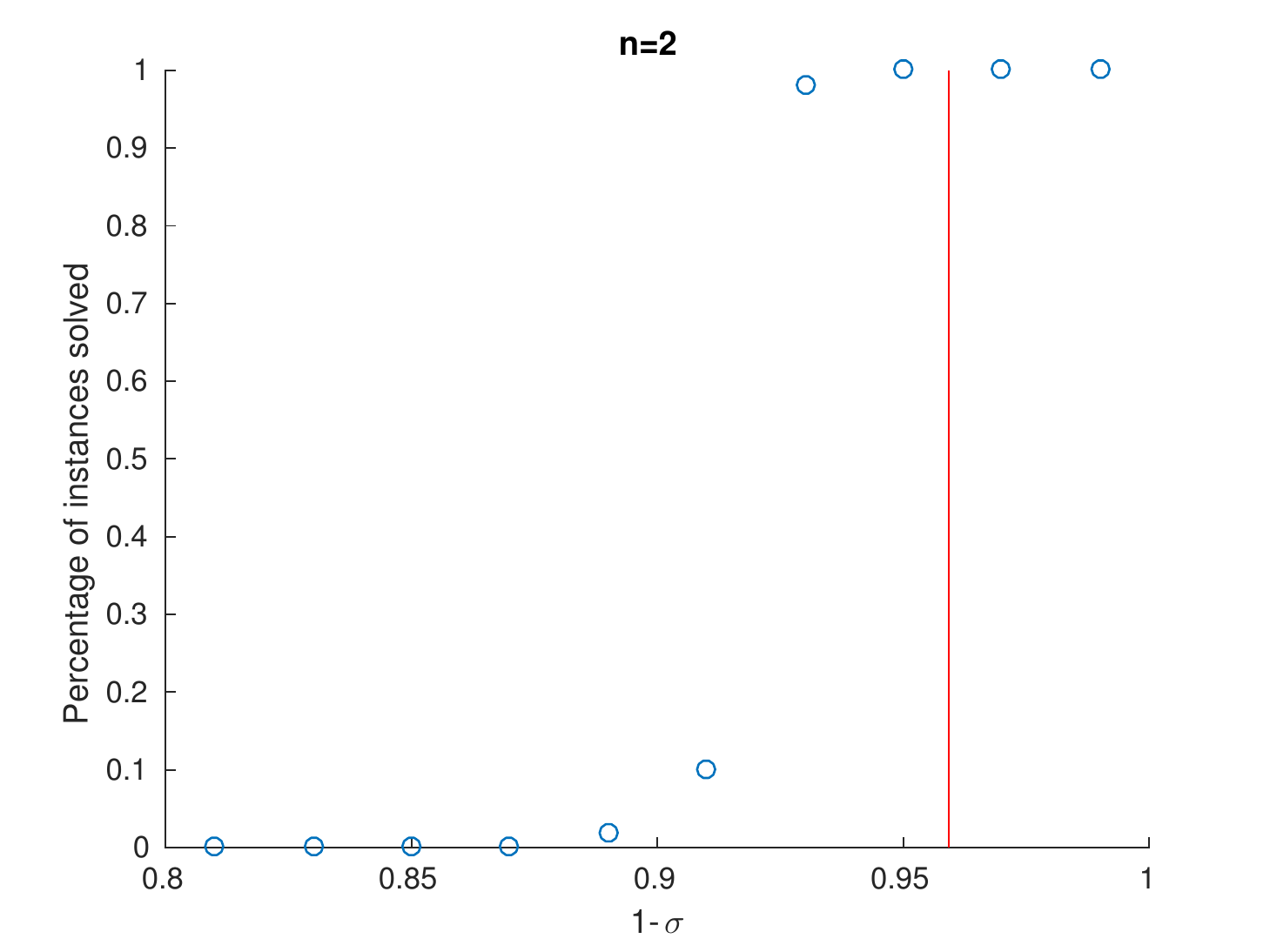}
\includegraphics[height=6cm]{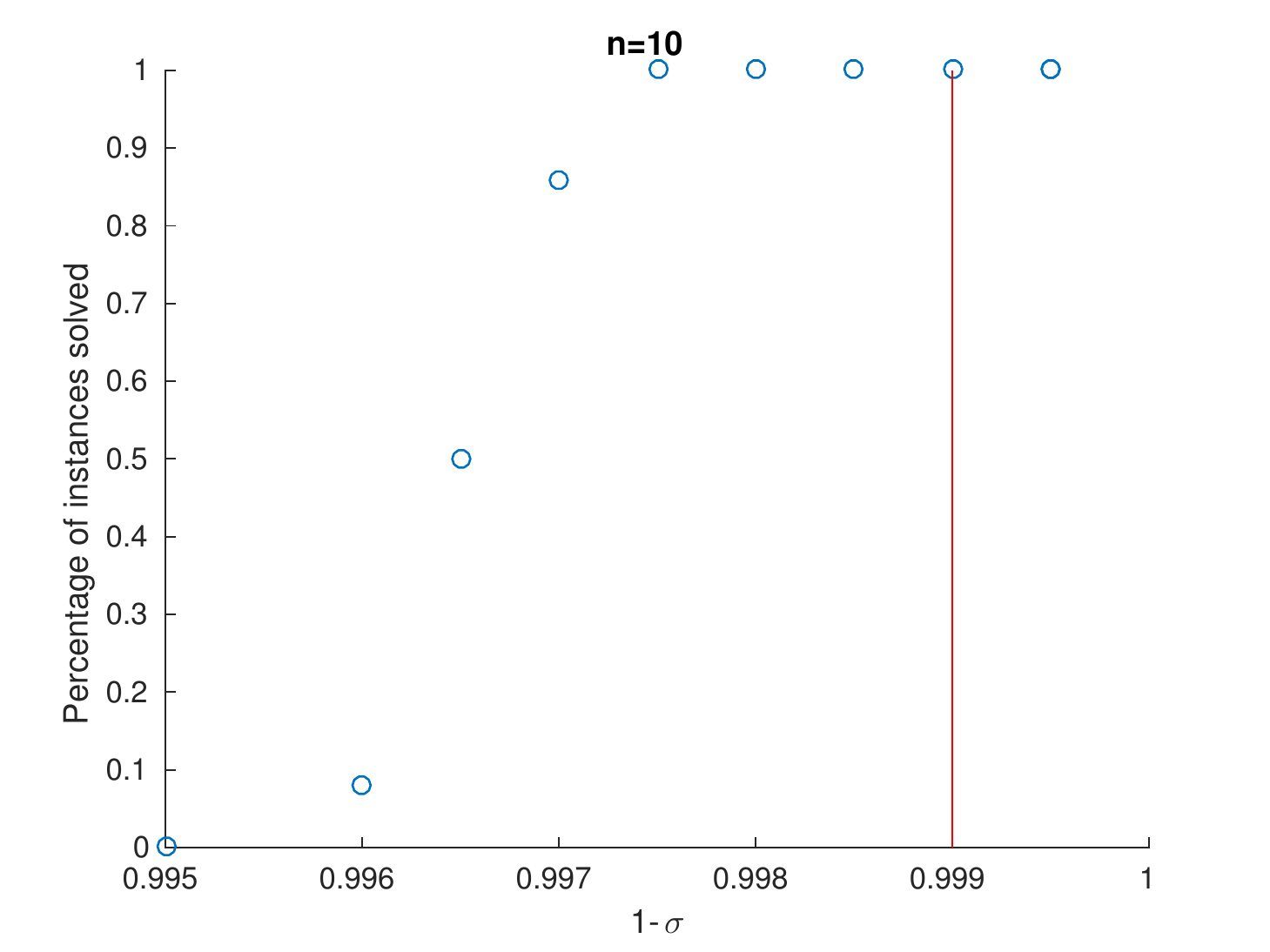}
\end{center}
\caption{Minimizing a simple quadratic function (dimension $n=2$ in the left, $n=10$ in the right) where with probability $1-\sigma$, a coordinate is computed incorrectly.} 
\end{figure}
\subsection{Stochastic gradient based method comparison}
In this subsection, we show how STORM applies to empirical risk minimization in machine learning.
Consider a training dataset of $N$ samples, $\{(x_i,y_i)\}_{i=1\dots N}$, where
$x_i\in \mathbb{R}^m$ is a vector of $m$ real-valued and $y_i\in\{-1,1\}$ indicates a positive or negative label respectively.
We will train a linear classifier and bias term $(w,\beta)\in\mathbb{R}^{m+1}$ by minimizing the smooth (convex) 
regularized logistic loss

$$f(w,\beta) =\frac{1}{N} \displaystyle\sum_{i=1}^N f_i(w,\beta)+\lambda\|w\|^2=\frac{1}{N}
 \displaystyle\sum_{i=1}^N \log(1 + \exp(-y_i(w^{T}x_i + \beta)))+\lambda\|w\|^2.$$

As in the typical machine learning setting, we will assume that $N >> m$  and computing $f(w,\beta)$ as well as 
$\nabla f(w,\beta)$ and $\nabla ^2 f(w,\beta)$ is prohibitive. Hence we will 
only compute estimates of these quantities by considering a sample $I \subset \{1,\dots,N\}$ of size
$|I| = n << N$, yielding

$$f_{I}(w,\beta) =\frac{1}{|I|}  \displaystyle\sum_{i\in I} f_i(w,\beta)+\lambda\|w\|^2, \ 
\nabla f_{I}(w,\beta) =\frac{1}{|I|}  \displaystyle\sum_{i\in I} \nabla f_i(w,\beta)+2\lambda w,$$
$$  \nabla^2f_{I}(w,\beta) =\frac{1}{|I|}  \displaystyle\sum_{i\in I} \nabla^2 f_i(w,\beta)+2\lambda I_n.$$
Thus we can construct models based on sample gradient and Hessians information and
we present the appropriate variant of  STORM as Algorithm \ref{STORMlog} in the Appendix. 

We compare  Algorithm \ref{STORMlog}  with   the well-known implementation
of Adagrad from the Ada-whatever package described in \cite{duchihazansinger}. We compare against this particular solver because
it is a well-understood stochastic gradient method used by the machine learning community that, like our algorithm, takes
adaptive step sizes, but unlike our algorithm, does not compute estimates of the loss function, but only computes averaged
stochastic gradients. 
For the choice of initialization in Algorithm \ref{STORMlog}, the following parameters were used: 
$\delta_{\max} = 10, \delta_0=1, x_0 = 0, \gamma=2,\eta_1=0.1,\eta_2=0.001,p_{\min} = m + 2, p_{\max} = N$.
Adagrad was also given the same initial point and an initial step size of $\delta_0=1$.

We implemented two versions of Algorithm \ref{STORMlog}: one which uses stochastic Hessians, and a second where
we do not compute stochastic Hessians, effectively setting $H_k=0$ on each iteration, yielding a trivial subproblem in the
step calculation. 

\begin{figure}
 \includegraphics[height=5cm]{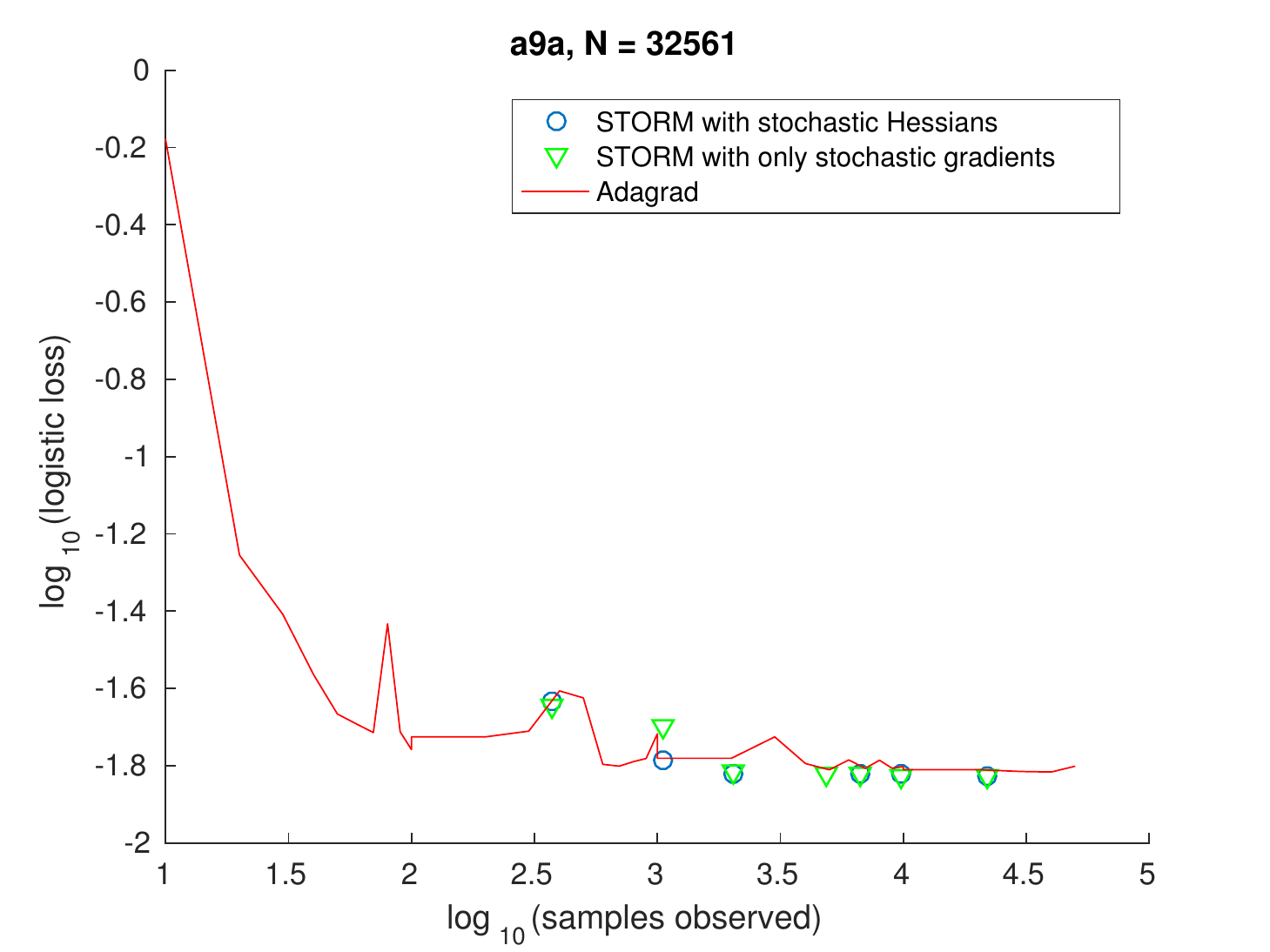} \includegraphics[height=5cm]{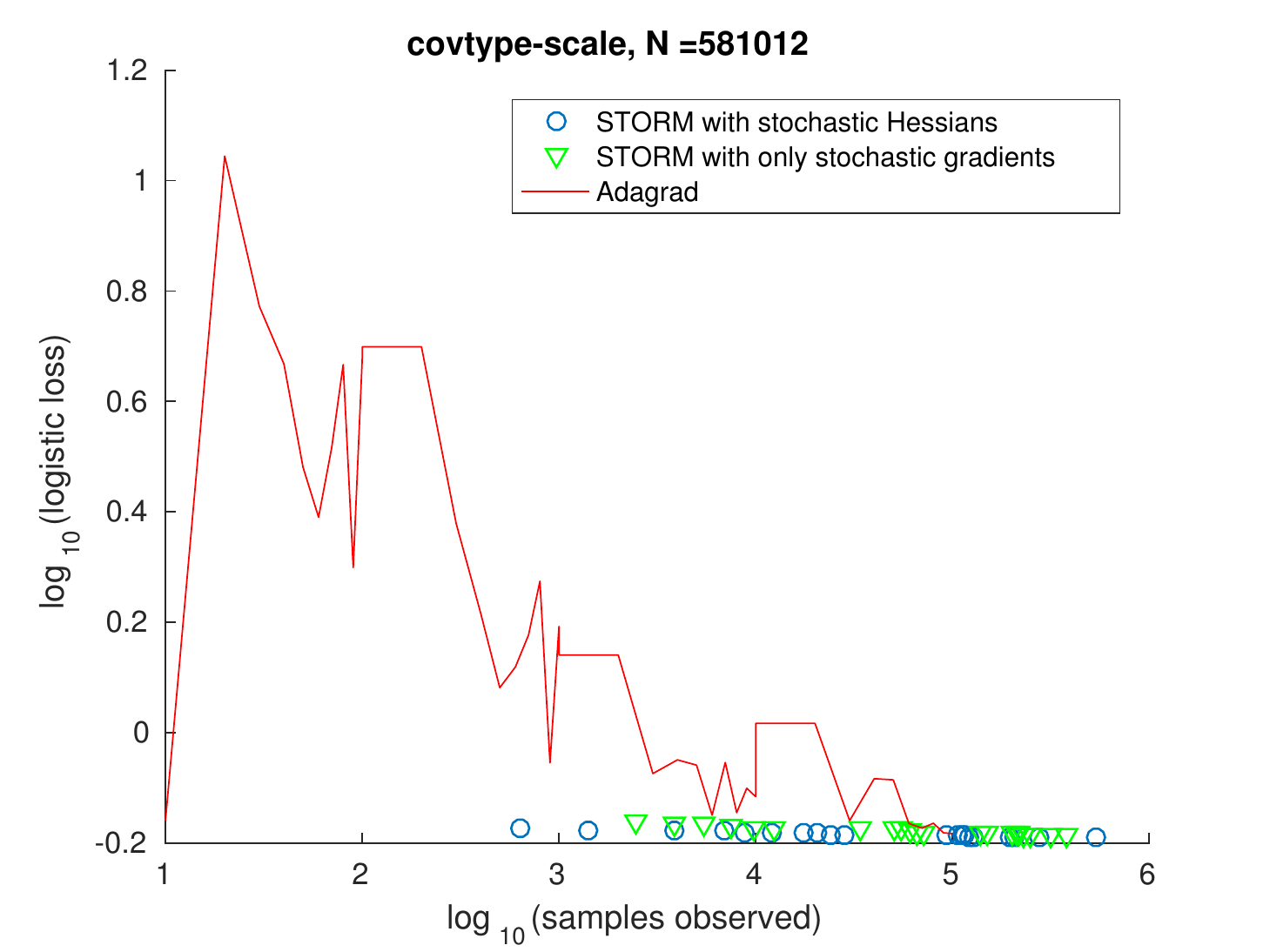}\\
 
 \includegraphics[height=5cm]{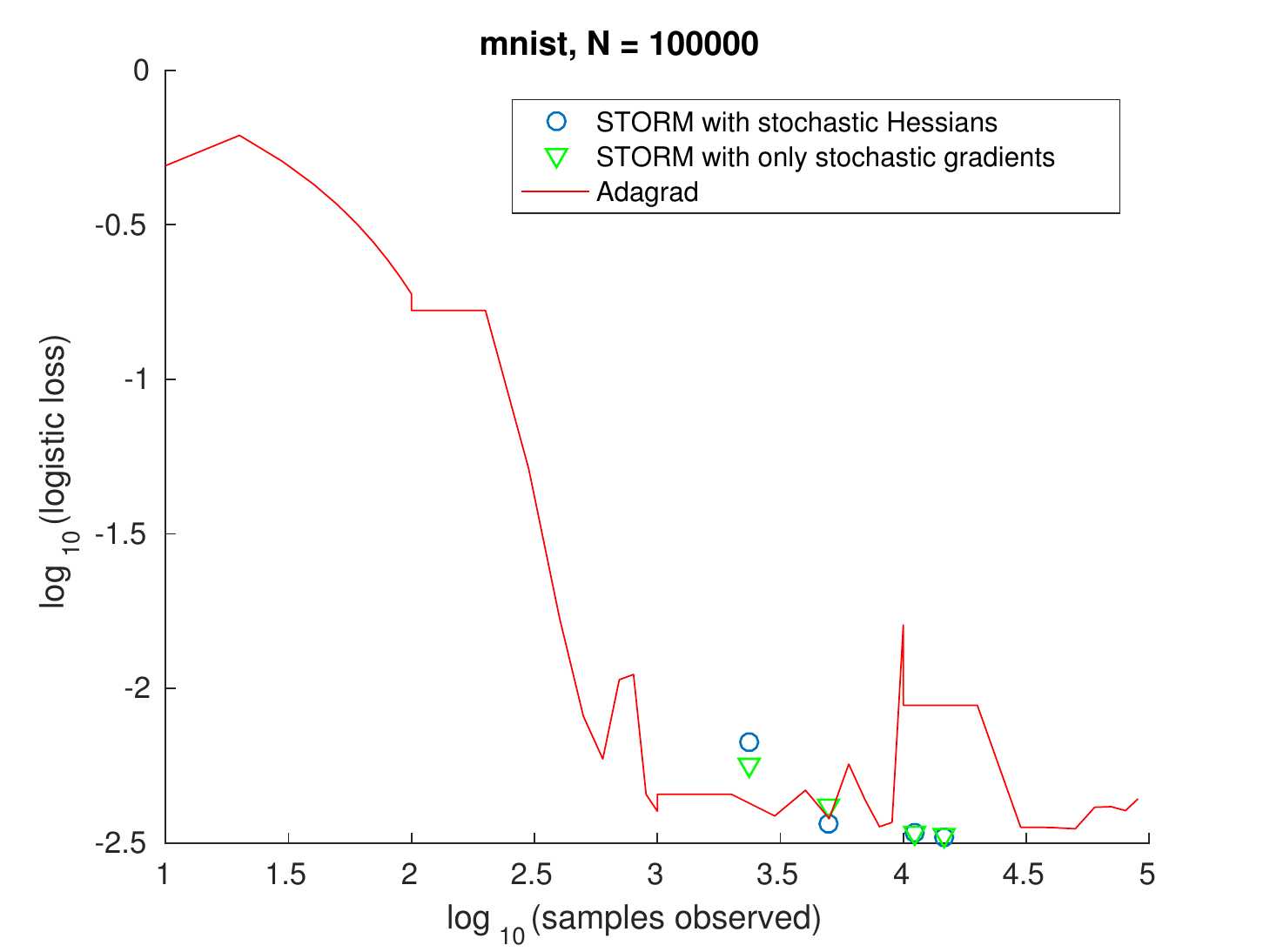} \includegraphics[height=5cm]{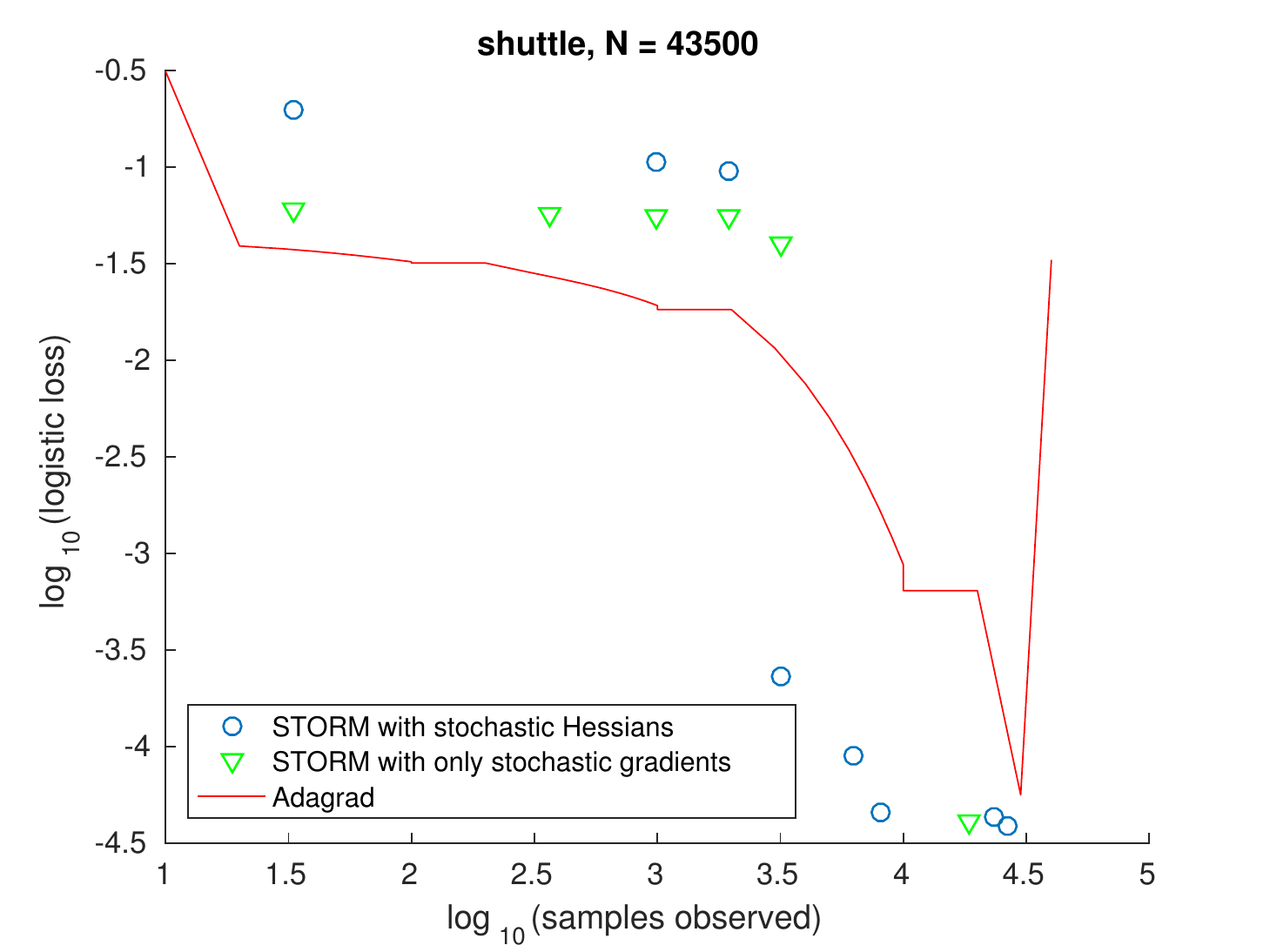}
 \caption{Trajectory of training logistic loss on four datasets.}
 \label{fig:trainingerror}
\end{figure}

\begin{figure}
  \includegraphics[height=5cm]{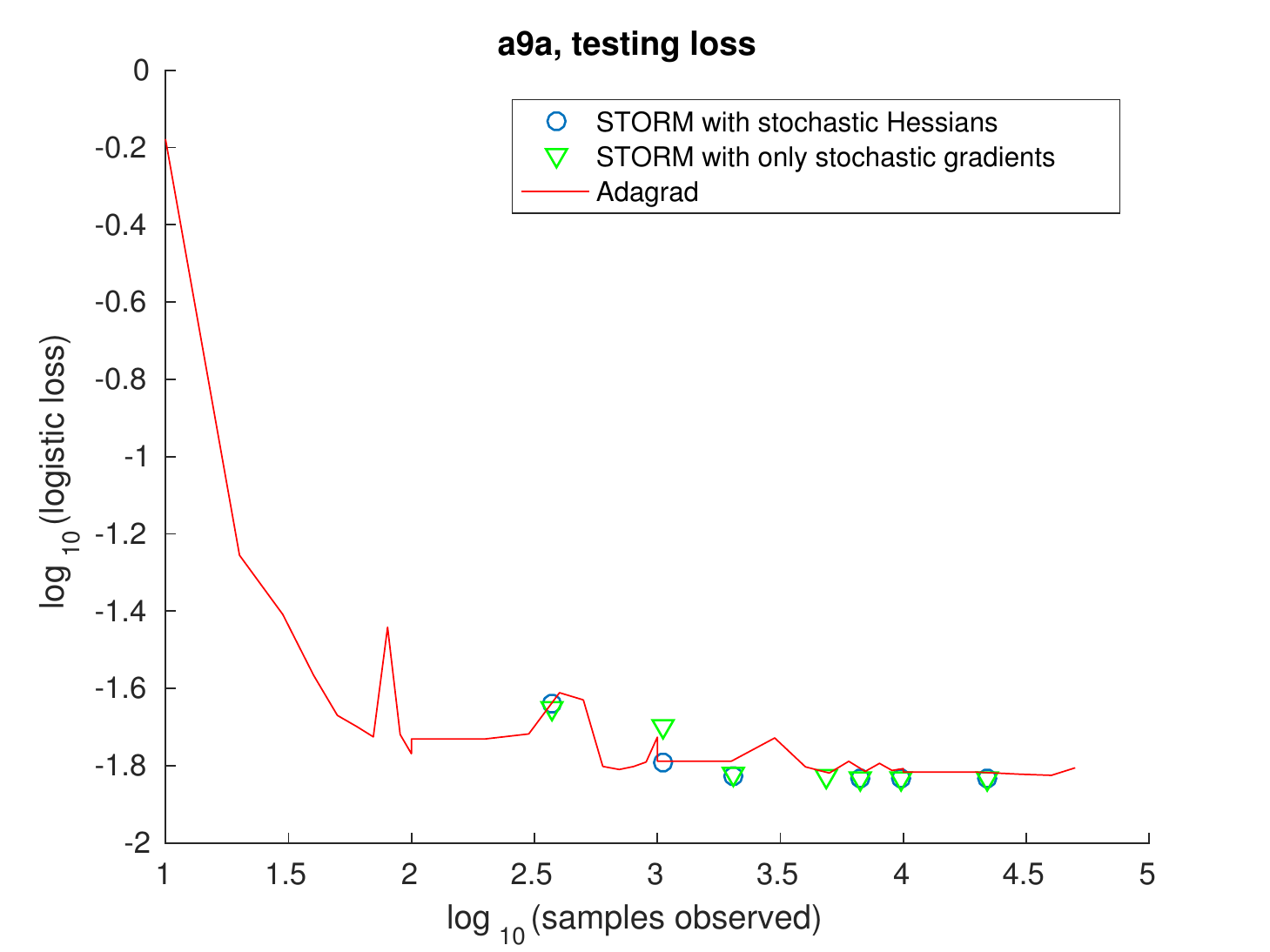} \includegraphics[height=5cm]{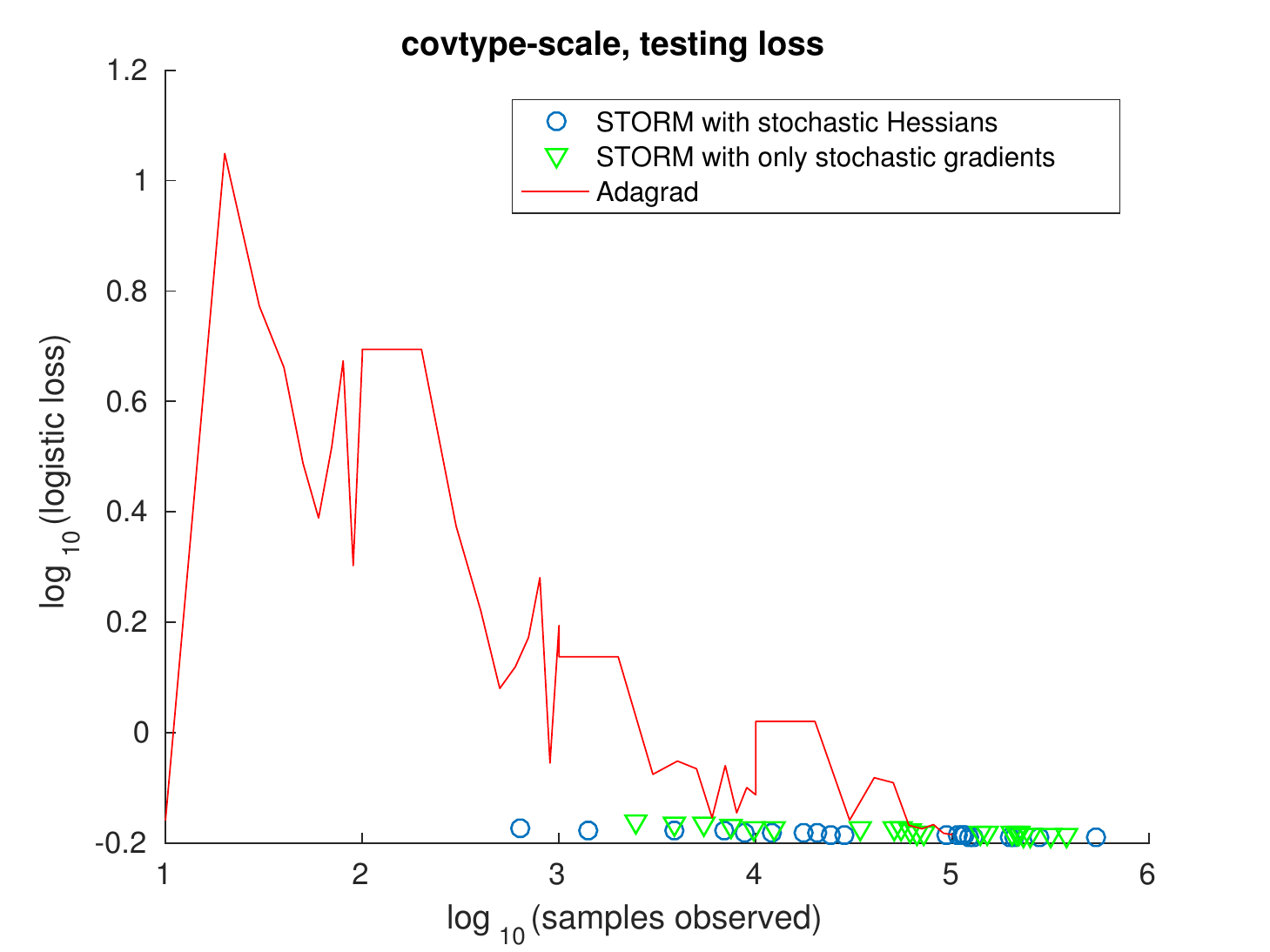}\\
 
 \includegraphics[height=5cm]{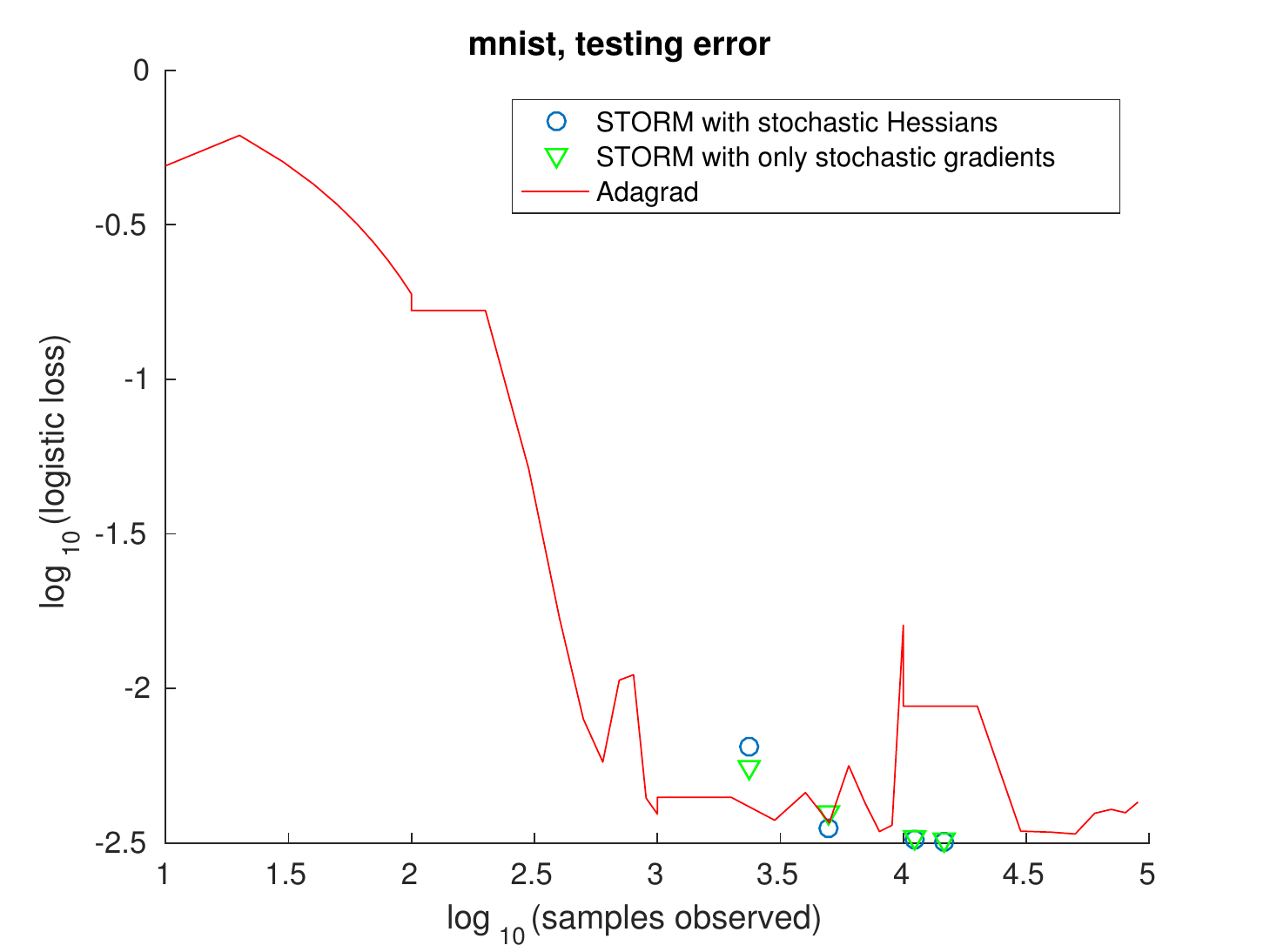} \includegraphics[height=5cm]{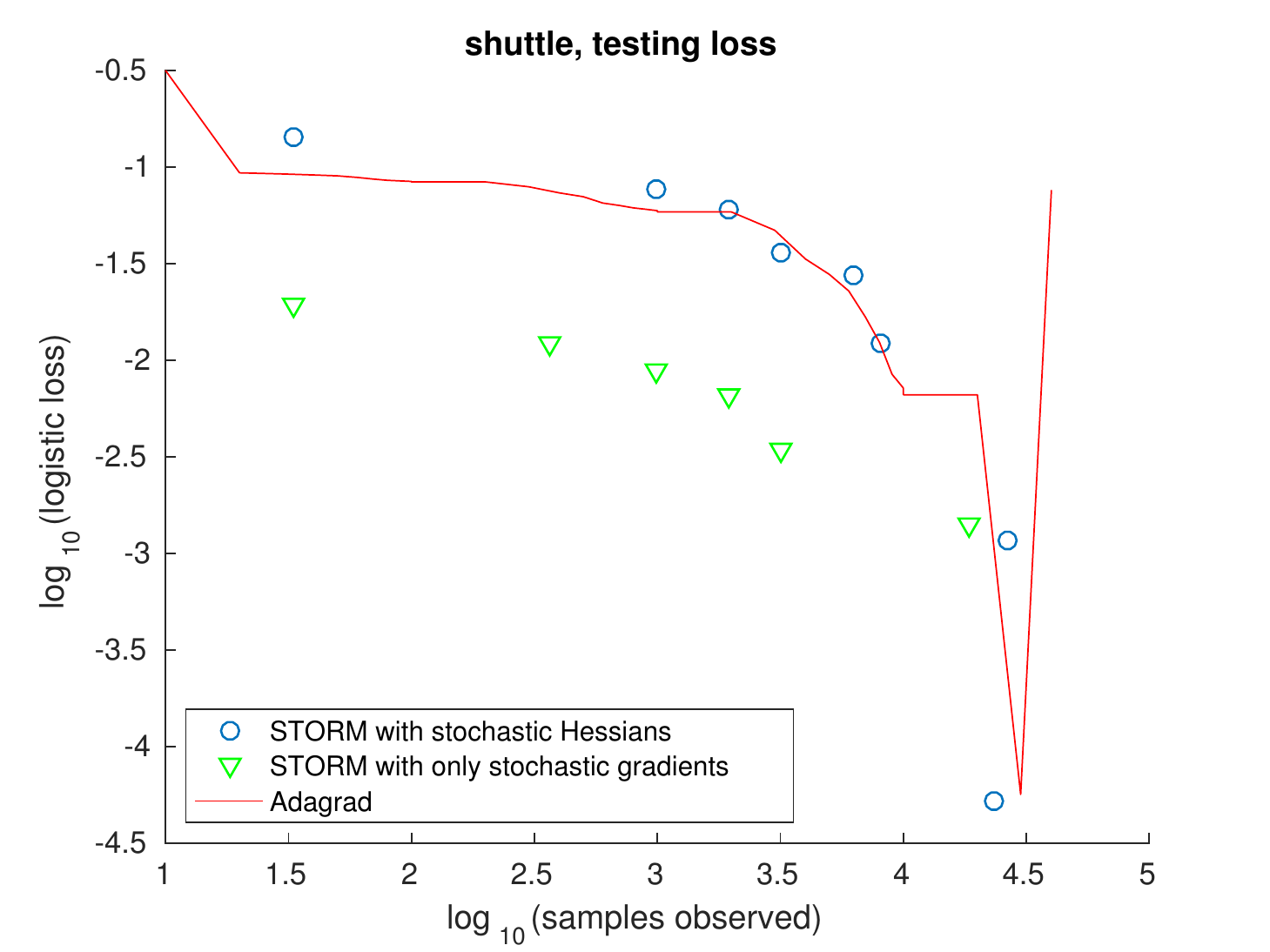}
 \caption{Trajectory of testing logistic loss on four datasets.}
 \label{fig:testingerror}
\end{figure}

For each of the datasets, we randomly partition $N$ into a training set of size $\lfloor 0.95*N\rfloor$ and
 a testing set of size $\lceil 0.05*N\rceil$. 
We set the maximum budget of data evaluations for each solver equal to the size of the training set, thus comparing 
various solvers' performance with a budget of roughly one full pass through the dataset.
In the two implementations of STORM, we plot in \ref{fig:trainingerror} the true training loss function value at the end of each successful iteration, while for Adagrad,
we simply plot the true training loss function value over an evenly spaced array of function value counts. 
Likewise in \ref{fig:testingerror}, we plot at the same points the value of the holdout testing loss function value.

Notice that, as expected, the true function values produced by Adagrad can vary widely over this horizon, but implementations
of STORM tend to yield fairly stable decreasing trajectories over its successful iterations. Also, we see that
the loss seems to generalize fairly well to the holdout test data. 

\bibliographystyle{plain}
\bibliography{references}

\section{Appendix}

\begin{algorithm}
\label{tr-saa}
  \SetAlgoNlRelativeSize{-5}
  \SetKw{break}{break}
  \SetKw{true}{true}
  \SetKw{goto}{go to}
 \textbf{(Initialization):} Choose an initial point $x_0$ and an initial trust-region radius $\d_0\in (0,\d_{\max})$ with 
    $\d_{\max}>0$. Choose  constants $\gamma>1$, $\eta_1\in(0,1)$, $\eta_2\in(0,\infty)$, $p_{\max}=(n+1)(n+2)/2$, $p_{\min}\geq n+1$. Set $k \gets 0$. 
    Select some initial interpolation set $Y_0\subset B(x_k,\d_k)$ so that $|Y_0|\leq p_{\max}$ and $x_0 \in Y_0$. 
    Compute an averaged function value estimate $\bar f(y) = \frac{1}{p_{min}}\sum_{i=1}^{p_{\min}} \tilde f(y,\omega_i)$ at each $y\in Y_0$. \\
    \While{\true}{
\textbf{(Update sample rate):} Set sample rate $p_k = \max\{p_{\min} + k, 1/\d\}.$\\

\textbf{(Update interpolation value estimates):} For each $y\in Y_k\cap Y_{k-1}$, compute 
$\bar f(y)= \frac{1}{p_k}[\sum_{i=p_{k-1}+1}^{p_k} \tilde f(y,\omega_i)+\bar f(y)]$ and for each
$y\in Y_k\setminus Y_{k-1}$, compute 
$\bar f(y)= \frac{1}{p_k}\sum_{i=1}^{p_k} \tilde f(y,\omega_i)$ .\\

\textbf{If the algorithm is TR-SAA-resample,} for each $y\in Y_k$, compute
$\bar f(y) = \frac{1}{p_k}\sum_{i=1}^{p_k} \tilde f(y,\omega_i)$.

\textbf{(Model building):} Build a quadratic model $m_k(x_k+s)=f_k+g_k^\top s+s^\top H_k s$  with $s=x-x_k$ that interpolates $\bar f(y)$ 
at the points of $Y_k$.\\

\textbf{(Step calculation):} Compute $s_k = \arg \underset{s: \| s \|\le \d_k}{\min}  m_k(s) $ (approximately) so that $s_k$ 
satisfies \eqref{eqn:CS}.\\

\textbf{(Estimate calculation):} Compute new estimate $f_k^s =\frac{1}{p_k} \sum_{i=1}^{p_k} \tilde f(x_k + s_k,\omega_i)$ of $f(x_k+s_k)$.
Denote the current estimate of $f(x_k)$ by $f_k^0$.\\

\textbf{(Acceptance of the trial point):} Compute $\rho_k = \dfrac{f_k^0-f_k^s}{m_k(x_k)-m_k(x_k+s_k)}.$\\ 
If $\rho_k\ge \eta_1$ and $\|g_k\|\geq \eta_2\delta_k$, then $x_{k+1} \gets x_k+s_k$; otherwise, $x_{k+1}\gets x_k$.\\

\textbf{(Trust-region radius update):} If $\rho_k\ge \eta_1$, $\d_{k+1} \gets\min\{  \gamma \d_k,\d_{\max}\}$; 
otherwise $\d_{k+1}\gets\gamma^{-1} \d_k$.\\

\textbf{(Interpolation set update):} Augment $Y_k$ with $x_k+s_k$.
If $|Y_k| > p_{\max}$, remove the point of $Y_k$ furthest from $x_{k+1}$.\\

\textbf{(Iterate):} $k \gets k+1$.
}
  \caption{TR-SAA}
\end{algorithm}

\begin{algorithm}
\label{STORM-addrel}
  \SetAlgoNlRelativeSize{-5}
  \SetKw{break}{break}
  \SetKw{true}{true}
  \SetKw{goto}{go to}
 \textbf{(Initialization):} Choose an initial point $x_0$ and an initial trust-region radius $\d_0\in (0,\d_{\max})$ with 
    $\d_{\max}>0$. Choose  constants $\gamma>1$, $\eta_1\in(0,1)$, $\eta_2\in(0,\infty)$, $p_{\min}$; Set $k \gets 0$. 
    Select a regression set $Y_0 \subset B(x_k,\d_k)$ satisfying $|Y_0| = p_{\min}$. 
     \\
    \While{\true}{
\textbf{(Update sample rate):} Choose a sample rate $p_k = \max\{p_{\min} + k, 1/\d\}.$\\

\textbf{(Regression set update):} Uniformly sample a regression set $Y_k \subset B(x_k,\d_k)$ satisfying $|Y_k| = p_k$.\\

\textbf{(Compute new regression value estimates):} For each $y\in Y_k$, compute a single estimate $\tilde f(y,\omega)$.\\

\textbf{(Model building):} Build a quadratic model $m_k(x_k+s)=f_k+g_k^\top s+s^\top H_k s$  with $s=x-x_k$ that regresses $\tilde f(y)$ 
at the points of $Y_k$.\\

\textbf{(Step calculation):} Compute $s_k = \arg \underset{s: \| s \|\le \d_k}{\min}  m_k(s) $ (approximately) so that $s_k$ 
satisfies \eqref{eqn:CS}.\\

\textbf{(Estimates calculation):} Compute new estimates $f_k^0 = \sum_{i=1}^{p_k} \tilde f(x_k,\omega_i)$ and $f_k^s = \sum_{i=1}^{p_k} \tilde f(x_k + s_k,\omega_i)$
of $f(x_k)$ and $f(x_k+s_k)$.\\

\textbf{(Acceptance of the trial point):} Compute $\rho_k = \dfrac{f_k^0-f_k^s}{m_k(x_k)-m_k(x_k+s_k)}.$\\ 
If $\rho_k\ge \eta_1$ and $\|g_k\|\geq \eta_2\delta_k$, then $x_{k+1} \gets x_k+s_k$; otherwise, $x_{k+1}\gets x_k$.\\

\textbf{(Trust-region radius update):} If $\rho_k\ge \eta_1$, $\d_{k+1} \gets\min\{  \gamma \d_k,\d_{\max}\}$; 
otherwise $\d_{k+1}\gets\gamma^{-1} \d_k$.\\

\textbf{(Iterate):} $k \gets k+1$.
}
  \caption{STORM for unbiased noise}
\end{algorithm}

\begin{algorithm}
\label{STORM-failure}
  \SetAlgoNlRelativeSize{-5}
  \SetKw{break}{break}
  \SetKw{true}{true}
  \SetKw{goto}{go to}
 \textbf{(Initialization):} Choose an initial point $x_0$ and an initial trust-region radius $\d_0\in (0,\d_{\max})$ with 
    $\d_{\max}>0$. Choose  constants $\gamma>1$, $\eta_1\in(0,1)$, $\eta_2\in(0,\infty)$, $p_0\geq n+1$, $p_{\max}=(n+1)(n+2)/2$. Set $k \gets 0$. 
    Select an interpolation set $Y_0 \subset B(x_k,\d_k)$ satisfying $|Y_0| = p_0$. 
     \\
    \While{\true}{

\textbf{(Compute new interpolation value estimates):} For each $y\in Y_k$, compute a (new) estimate $\tilde f(y,\omega)$.\\

\textbf{(Model building):} Build a quadratic model $m_k(x_k+s)=f_k+g_k^\top s+s^\top H_k s$  with $s=x-x_k$ that interpolates $\tilde f(y)$ 
at the points of $Y_k$.\\

\textbf{(Step calculation):} Compute $s_k = \arg \underset{s: \| s \|\le \d_k}{\min}  m_k(s) $ (approximately) so that $s_k$ 
satisfies \eqref{eqn:CS}.\\

\textbf{(Estimates calculation):} Compute new estimates $f_k^0 = \tilde f(x_k,\omega_i)$ and $f_k^s = \tilde f(x_k + s_k,\omega_i)$
of $f(x_k)$ and $f(x_k+s_k)$.\\

\textbf{(Acceptance of the trial point):} Compute $\rho_k = \dfrac{f_k^0-f_k^s}{m_k(x_k)-m_k(x_k+s_k)}.$\\ 
If $\rho_k\ge \eta_1$ and $\|g_k\|\geq \eta_2\delta_k$, then $x_{k+1} \gets x_k+s_k$; otherwise, $x_{k+1}\gets x_k$.\\

\textbf{(Trust-region radius update):} If $\rho_k\ge \eta_1$, $\d_{k+1} \gets\min\{  \gamma \d_k,\d_{\max}\}$; 
otherwise $\d_{k+1}\gets\gamma^{-1} \d_k$.\\

\textbf{(Interpolation set update):} Augment $Y_k$ with $x_k+s_k$.
If $|Y_k| > p_{\max}$, remove the point of $Y_k$ furthest from the new trust region center $x_{k+1}$.

\textbf{(Iterate):} $k \gets k+1$.
}
  \caption{STORM for unbiased noise}
\end{algorithm}

\begin{algorithm}
\label{STORMlog}
  \SetAlgoNlRelativeSize{-5}
  \SetKw{break}{break}
  \SetKw{true}{true}
  \SetKw{goto}{go to}
 \textbf{(Initialization):} Choose an initial point $x_0 = (w_0,\beta_0)$ and an initial trust-region radius $\d_0\in (0,\d_{\max})$ with 
    $\d_{\max}>0$. Choose  constants $\gamma>1$, $\eta_1\in(0,1)$, $\eta_2\in(0,\infty)$, $p_0$, $p_{\max}$; Set $k \gets 0$.  
     \\
    \While{\true}{

\textbf{(Determine sample rate):} Choose a sample rate $p_k$. In our implementation, we will use 
$p_k = \min\{p_{\max},\max\{100*k + p_0, \lceil 1/\d_k^2\rceil\}\}$. \\

\textbf{(Model building):} Uniformly (without replacement) draw a sample $I_k\subset \{1,\dots,N\}$. Compute a stochastic gradient 
$g_k = \nabla f_{I_k}(w,\beta)$ and stochastic Hessian $H_k = \nabla^2 f_{I_k}(w,\beta)$. Define a quadraticmodel
$m_k(s) = g_k^Ts + \frac{1}{2}s^TH_k s$.
and stochastic Hessian \\

\textbf{(Step calculation):} Compute $s_k = \arg \underset{s: \| s \|\le \d_k}{\min}  m_k(s) $ (approximately) so that $s_k$ 
satisfies \eqref{eqn:CS}.\\

\textbf{(Estimates calculation):} Draw new samples $I_k^0, I_k^s$ and compute estimates
$f_k^0 = f_{I_k^0}(x_k)$ and $f_k^s = f_{I_k^s}(x_k+s_k)$
of $f(x_k)$ and $f(x_k+s_k)$, respectively.\\

\textbf{(Acceptance of the trial point):} Compute $\rho_k = \dfrac{f_k^0-f_k^s}{m_k(x_k)-m_k(x_k+s_k)}.$\\ 
If $\rho_k\ge \eta_1$ and $\|g_k\|\geq \eta_2\delta_k$, then $x_{k+1} \gets x_k+s_k$; otherwise, $x_{k+1}\gets x_k$.\\

\textbf{(Trust-region radius update):} If $\rho_k\ge \eta_1$, $\d_{k+1} \gets\min\{  \gamma \d_k,\d_{\max}\}$; 
otherwise $\d_{k+1}\gets\gamma^{-1} \d_k$.\\

\textbf{(Iterate):} $k \gets k+1$.
}
  \caption{STORM for minimizing logistic loss}
\end{algorithm}

\end{document}